\documentclass[a4paper,reqno]{amsart}

\usepackage{amsmath}
\usepackage{amssymb}
\usepackage{amsthm}
\usepackage{latexsym}
\usepackage{amscd}
\usepackage{xypic}
\xyoption{curve}
\usepackage{ifthen} 
\usepackage{hyperref} 
\usepackage{graphicx}
\usepackage{enumerate}
\usepackage{enumitem}
\usepackage{rotating}
\usepackage{xcolor}
\usepackage{mathtools}
\usepackage{mathrsfs}
\usepackage{todonotes}
\usepackage{hyperref}
\usepackage{amsrefs}

\usepackage{float}
\restylefloat{table}

\numberwithin{equation}{section}

\def\cocoa{{\hbox{\rm C\kern-.13em o\kern-.07em C\kern-.13em o\kern-.15em A}}}

\newtheorem*{mthm}{Main Theorem 1}
\newtheorem*{mthm2}{Main Theorem 2}
\newtheorem{theorem}{Theorem}[section]

\newtheorem{lemma}[theorem]{Lemma}
\newtheorem{proposition}[theorem]{Proposition}

\theoremstyle{definition}
\newtheorem{remark}[theorem]{Remark}
\newtheorem{definition}[theorem]{Definition}

\newtheorem{open problem}{Open Problem}

\newcommand {\sHom}{\mathcal{H}\kern -0.25ex{\mathit om}}
\newcommand {\sExt}{\mathcal{E}\kern -0.25ex{\mathit xt}}
\newcommand {\sTor}{\mathcal{T}\kern -0.25ex{\mathit or}}

\newcommand {\rk}{\mathrm{rk}}

\newcommand {\supp}{\mathrm{Supp}}

\newcommand {\Ext}{\mathrm{Ext}}
\newcommand {\Hom}{\mathrm{Hom}}
\newcommand {\Hilb}{\mathcal{H}\kern -0.25ex{\mathit ilb\/}}

\newcommand{\inext}{{\mathcal E}{\it xt}}

\newcommand {\cU}{\mathcal{U}}

\newcommand{\cE}{{\mathcal E}}
\newcommand{\cEv}{{\mathcal E}^{\vee}}
\newcommand{\cF}{{\mathcal F}}
\newcommand{\cM}{{\mathcal M}}
\newcommand{\cN}{{\mathcal N}}
\newcommand{\cO}{{\mathcal O}}
\newcommand{\cG}{{\mathcal G}}
\newcommand{\cT}{{\mathcal T}}
\newcommand{\cZ}{{\mathcal Z}}
\newcommand{\cI}{{\mathcal I}}

\newcommand {\bZ}{\mathbb{Z}}

\newcommand {\bP}{\mathbb{P}}

\newcommand {\bF}{\mathbb{F}}

\newcommand{\Pic}{\operatorname{Pic}}

\newcommand{\lra}{\longrightarrow}

\def\p#1{{\bP^{#1}}}

\def\H#1{\mathrm{H}^{#1}}
\def\h#1{\mathrm{h}^{#1}}

		\title[Instanton Sheaves on Ruled Fano 3-folds]{Instanton Sheaves on \\ Ruled Fano 3-folds of Picard Rank 2 and Index 1}
		
		
		\author[O. Genc]{Ozhan Genc}
		\author[M. Jardim]{Marcos Jardim}
		
\begin{document}

\begin{abstract}
	We study rank 2 $h$-instanton sheaves on projective threefolds. We demonstrate that any orientable rank 2, non-locally free $h$-instanton sheaf with defect 0 on a threefold can be obtained as an elementary transformation of a locally free $h$-instanton sheaf. Our focus then shifts to ruled Fano threefolds of Picard rank 2 and index 1, of which there are five deformation classes. We establish the existence of orientable rank 2 $h$-instanton bundles on such varieties. Additionally, we prove the existence of Ulrich bundles on such varieties, which correspond to $h$-instanton sheaves of minimum charge.
\end{abstract}

\maketitle

\begin{quote}
	\leftskip=70pt {\it To our friend Gianfranco Casnati}
\end{quote}

\section{Introduction}

In their fundamental work \cite{ADHM}, Atiyah, Drinfeld, Hitchin, and Manin presented the notion of \textit{mathematical instanton bundles} as rank 2 holomorphic vector bundles on $\p3$ that correspond to anti-self-dual connections, a.k.a. instantons, on the 4-dimensional round sphere $S^4$. To be precise, a \textit{mathematical (or complex) instanton bundle} of charge $n$ was defined in \cite[p. 185]{OSSG80} as a stable rank 2 vector bundle $E$ with Chern classes $c_1(E)=0, \ c_2(E)=n$ and such that $\H{i}(E(-2))=0$ for $i=1,2$.

In the following decades, several authors presented different generalizations of mathematical instanton bundles, first to odd-dimensional projective spaces \cite{MCS}, then to non-locally free sheaves of any rank on arbitrary projective spaces \cite{J-inst}, and then to other Fano three-dimensional varieties, see \cites{F,K,CCGM21,CJ}. More recently, Antonelli and Casnati, following up on previous work by Antonelli--Malaspina \cite{AM22}, proposed in \cite{AC23} the notion of \textit{$h$-instanton sheaves} on arbitrary projective varieties.

In another direction, the notion of rank 0 instanton sheaves on $\bP^3$ was initially introduced in \cite[Definition 6.1]{HaLa} and further studied in \cite{GJ} and more recently in \cite{CJ} in the context of Fano threefolds with Picard rank 1. In particular, \cite[Main Theorem]{GJ} provides a precise relation between non-locally free instanton sheaves and rank 0 instanton sheaves: namely, if $\cE$ is a non-locally free rank 2 instanton sheaf on $\bP^3$, then $\cE^{\vee\vee}$ is a locally free rank 2 instanton sheaf and $\cE^{\vee\vee}/\cE$ is a rank 0 instanton sheaf. This result was then generalized to the context of Fano threefolds with Picard rank 1, see \cite[Theorem 24]{CJ}.

More precisely, let $X$ be an irreducible projective 3-dimensional scheme endowed with an ample, globally generated line bundle $\cO_X(h)$. An \emph{$h$-instanton sheaf with defect 0} on $X$ is a torsion-free sheaf $\cE$ satisfying the following cohomological vanishing conditions
$$ \h0 \big(\cE(-h)\big) = \h{1} \big(\cE(-2h)\big)=\h{2} \big(\cE(-2h)\big) = \h{3} \big(\cE(-3h)\big) = 0. $$
This is a special case of the definition introduced in \cite[Definition 1.3 and Theorem 1.4]{AC23}, where the authors also introduce $h$-instantons with defect 1. Following \cites{AC23,AM22,J-inst}, but departing from \cites{CJ,F,K}, the definition here considered does not assume that the $\cE$ is slope-(semi)stable; a comparison between these definitions is considered in Section \ref{sec:h-instantons}.

Our first goal in the present paper is to introduce the concept of \textit{rank 0 $h$-instanton sheaves} on three-dimensional projective varieties: this is a 1-dimensional sheaf $\cT$ satisfying $h^{p}(\cT(-2h))=0$ for all $p$, see Definition \ref{defn: rank 0 instanton}  below. We then generalize the classification of non-locally free rank 2 $h$-instanton sheaves we mentioned above, showing that if $\cE$ is a non-locally free rank 2 $h$-instanton sheaf on $X$, then $\cE^{\vee\vee}$ is a locally free rank 2 $h$-instanton sheaf and $\cE^{\vee\vee}/\cE$ is a rank 0 $h$-instanton sheaf; further details in Theorem \ref{thm:double_dual_bundle} below.

In the second part of the paper, we will focus on the construction of rank 2 $h$-instanton sheaves on ruled Fano threefolds with Picard rank 2 and index 1. These varieties are of the form $X_c:=\bP\big( F_c \big)$ where $F_c$ is one of the rank 2 Fano bundles on $\bP^2$ provided in Table \ref{table: for F}, see \cite[Theorem in p. 296]{SW90}. In addition, we consider a polarization on $X_c$ of the form $h=\xi+f$ where $\xi = c_1(\cO_{X_{c}} (1))$ and $f=c_1(\pi^*\cO_{\bP^2}(1))$ with $X_c\stackrel{\pi}{\to}\p2$ is the projection induced from $F_c\to\p2$; $h$ is ample by \cite[Proposition 2.1]{SW90}. 


\begin{table}[h]
	\caption{Rank 2 Fano bundles on $\p2$ of index 1}
	\begin{tabular}{|c|c|}
		\hline
		$c=c_2$ & $F_c$                                                                                                                         \\ \hline
		0       & $F_0 \simeq \cO_{\p2} \oplus \cO_{\p2}(2)$                                                                                    \\ \hline
		1       & $F_1 \simeq \cO_{\p2}(1) \oplus \cO_{\p2}(1)$                                                                                 \\ \hline
		2       & $0 \rightarrow \cO_{\p2}(1) \rightarrow F_2 \rightarrow \cI_{pt|\p2}(1) \rightarrow 0$                                        \\ \hline
		3       & $0 \rightarrow \mathcal{O}_{\mathbb{P}^{2}}(-1)^2 \rightarrow \mathcal{O}_{\mathbb{P}^{2}}^{4} \rightarrow F_3 \rightarrow 0$ \\ \hline
		4       & $0 \to \mathcal{O}_{\mathbb{P}^{2}}(-2) \to \mathcal{O}_{\mathbb{P}^{2}}^{3} \to F_4 \to 0$                                   \\ \hline
	\end{tabular}
	\label{table: for F}
\end{table}

We recall that there exist two deformation classes of ruled Fano threefolds with Picard rank 2 and index 2, namely the $\bP\big({\rm T}\bP^2\big)$ (which coincides with the full flag manifold $\bF(0,1,2)$) and $\bP\big( \cO_{\p2} \oplus \cO_{\p2}(1)\big)$ (which coincides with the blow-up of $\bP^3$ at a point). Instanton sheaves on such threefolds have been studied in \cite{MML20} and \cite{CCGM21} respectively. Also, $h$-instanton bundles on $X_1:=\bP\big( F_1 \big)$ have been studied in \cite{AM22} by considering $X_1$ as a rational normal scroll.

Let $\mathcal{M}_X(2,c_1,c_2)$ denote the moduli space of $\mu$-stable rank 2 sheaves on a projective variety $X$ with Chern classes $c_1$ and $c_2$; in this paper, we prove the following:

\begin{mthm}\label{mthm}
	For each $\beta\ge2$ and $\alpha\ge6$ when $c\ge1$ and $\alpha\ge7$ when $c=0$, there exist rank 2, orientable, $\mu$-stable $h$-instanton bundles on $(X_c,h)$ with $c_2=\alpha \xi f+\beta f^2$ and charge $3\alpha+\beta+c-21$. Moreover, these instanton bundles are smooth points in a generically smooth irreducible component of dimension $10\alpha+4\beta+4c-21$ in $\mathcal{M}_{X_c}(2,2\xi+3f,\alpha \xi f+\beta f^2)$. 
	When $\alpha=6$, $\beta=2$ such bundles are Ulrich bundles on $X_1$.
\end{mthm}

In the last part of the paper, we focus on the case $\alpha = 5$, which is the minimum possible value, see Lemma \ref{lem:bounds_2ndChern} below. Using a completely different construction, we prove the following result.

\begin{mthm2}\label{2mthm}
	For each $l\ge1$, there exist rank 2, orientable, $\mu$-stable $h$-instanton bundles on $(X_c,h)$ with $c_2(\cG)=5\xi f+(l^2+2l+3-c)f^2$ and charge $(l-1)(l+3)$; in particular, $\cG$ is Ulrich when $l=1$. Moreover, these instanton bundles form a smooth, open subset of an irreducible component of dimension $4l(l+2)$ in $\mathcal{M}_{X_c}(2,2\xi+3f,5\xi f+(l^2+2l+3-c)f^2)$. 
\end{mthm2}

The paper is organized as follows. In Section \ref{sec:h-instantons}, we give the definition and some fundamental properties of $h$-instanton sheaves on threefolds along with an extended definition of rank 0 $h$-instanton sheaves. In Section \ref{sec:elem transf}, we establish the relation between locally free and non-locally free rank 2 orientable $h$-instanton sheaf with defect 0 on a smooth, irreducible projective threefolds. Section \ref{sec:ruled fanos} focuses on the detailed geometry of ruled Fano threefolds of index 1 and Picard rank 2. In Section \ref{sec:instantons on Xc}, we demonstrate the existence of rank 2 orientable $h$-instanton bundles with second Chern class $\alpha \xi f + 2 f^2$ for each $\alpha \ge 6$ when $c\ge1$ and $\alpha\ge7$ when $c=0$. In Section \ref{sec:deformation}, we apply the elementary deformation of the bundles constructed in Section \ref{sec:instantons on Xc} to prove Main Theorem 1. Finally, the last section focuses on the case $\alpha = 5$ and we establish Main Theorem 2, in which the minimum charge case (Ulrich bundles) is covered. We finish the paper with some open problems related to $h$-instanton bundles on $X_c$ which are not covered in this work.

\subsection*{Acknowledgments} The first author is supported by the grants MINIATURA 6 - 2022/06/X/ST1/01758 and MAESTRO NCN - UMO-2019/34/A/ST1/00263 - Research in Commutative Algebra and Representation Theory. The second author is partially supported by the CNPQ grant number 305601/2022-9, the FAPESP Thematic Project 2018/21391-1, and the FAPESP-ANR project 2021/04065-6. The first author extends his gratitude to the Department of Mathematics at the University of Campinas and to the second author for their hospitality during his visit.



\section{$h$-Instanton sheaves on threefolds} \label{sec:h-instantons}

Let $X$ be an irreducible projective scheme of dimension $3$ endowed with an ample, globally generated line bundle $\cO_X(h)$; we will often denote this pair by $(X,h)$. We will adopt the following definition for $h$-instanton sheaves on threefolds; it is a special case of the definition introduced in \cite{AC23}, compare with Definition 1.3 and Theorem 1.4 op. cit. 

\begin{definition}\label{defn:instanton}
	An \emph{$h$-instanton sheaf with defect 0} on $X$ is a torsion-free sheaf $\cE$ satisfying the following cohomological conditions
	\begin{enumerate}
		\item $\h0 \big(\cE(-h)\big)= \h{3} \big(\cE(-3h)\big)=0$;
		\item $\h{1} \big(\cE(-2h)\big)=\h{2} \big(\cE(-2h)\big)=0$;
		\item $\h1\big(\cE(-h)\big)=\h{2}\big(\cE(-3h)\big)$.
	\end{enumerate}
	The integer $k(\cE):=\h1\big(\cE(-h)\big)=\h{2}\big(\cE(-3h)\big)$ is called the \emph{charge} of $\cE$. 
\end{definition}

We will often omit the expression \textit{with defect 0} in the remainder of the text, regarding it as implicit in our results. We remark that, as a consequence of \cite[Definition 1.3 and Theorem 1.4]{AC23}, the following cohomological conditions are also satisfied:
$$ \h0 \big(\cE(-(t+1)h)\big) = \h{3} \big(\cE((t-3)h)\big) = 0 $$
$$ \h1 \big(\cE(-(t+2)h)\big) = \h{2} \big(\cE((t-2)h)\big) = 0 $$
for every $t\ge0$. In particular, $\chi(\cE(-h)\big)=-\h1\big(\cE(-h)\big)=-k(\cE)$. 

\begin{remark}
	Alternatively, one could define the charge of an $h$-instanton sheaf $\cE$ as $c_2(\cE)\cdot h$ or even just $c_2(\cE)$; these are related to $k(\cE)$ via the Grothendieck--Riemann--Roch formula. We opt for defining the charge as above because $k(\cE)\ge0$, and $k(\cE)=0$ if and only if $\cE$ is an Ulrich bundle.
\end{remark}

We note that Definition \ref{defn:instanton} generalizes Faenzi's \cite[Definition 1]{F} and Kuznetsov's \cite[Definition 1.1]{K} definitions for rank 2 instanton bundles on Fano threefolds of Picard rank 1, and further extended by Comaschi and Jardim \cite[Definition 6]{CJ} to encompass non-locally free sheaves of arbitrary rank. More precisely, assume that $X$ is a Fano threefold with $\Pic(X)=\bZ\cdot h$ (where $h$ is the ample generator) and index $i_X=2q_X+r_X$. Recall that an instanton sheaf on $X$, in the sense of \cite[Defintion 6]{CJ}, is a $\mu$-semistable torsion-free sheaf $\cF$ with $c_1(\cF)=-r_X\cdot\rk(\cF)/2$ and $\h1(\cF(-q_X\cdot h))=\h2(\cF(-q_X\cdot h))=0$; in particular, the rank of an instanton sheaf on a Fano threefold of odd index must be even.

Let us now compare \cite[Defintion 6]{CJ} with Definition \ref{defn:instanton}. If $\cF$ is an instanton sheaf, set $\cE:=\cF((2-q_X)\cdot h)$, First, note that 
$$ \h{p}(\cE(-2h)) = \h{p}(\cF(-q_X\cdot h) = 0 ~~ {\rm for} ~~ p = 1,2. $$
The fact that $\cF$ is $\mu$-semistable implies that
$$ \h0(\cE(-h)) = \h0(\cF((1-q_X)\cdot h)) = 0 ~~ {\rm and} ~~ \h3(\cE(-3h)) = \h3(\cF((-1-q_X)\cdot h)) = 0 $$
when $i_X=3,4$. Therefore, $\cE$ is a $h$-instanton sheaf when $X=\p3$ or $X$ is a 3-dimensional quadric in $\p4$. Otherwise, if $i_X=1,2$, then $\mu$-semistability is not enough to guarantee the vanishing in item (1) of Definition \ref{defn:instanton}. On the other hand, if $i_X=2$, then $\mu$-stable instanton sheaves in the sense of \cite[Defintion 6]{CJ} do satisfy item (1) of Definition \ref{defn:instanton}.

Notice that we have not fixed the first Chern class of an $h$-instanton sheaf, as usual in the various definitions of instantons sheaves available in the literature. 

\begin{definition}\label{defn:orientable}
	A torsion-free sheaf $\cE$ on $X$ is said to be \emph{orientable} if $c_1(\cE) = (4h + K_X)\rk(\cE)/2.$
\end{definition}

Note that instanton sheaves in the sense of \cite[Definition 6]{CJ} are orientable. 

\begin{lemma} \label{lem:rk2instanton}
	Let $\cE$ be an orientable, locally free sheaf of rank $2$. $\cE$ is an $h$-instanton if and only if $\h0\big(\cE(-h)\big) = \h{1}\big(\cE(-2h)\big)=0$.
\end{lemma}
\begin{proof}
	If  $\cE$ is an orientable rank 2 locally free sheaf, then $\cE^\vee=\cE(-4h-K_X)$, thus Serre duality yields
	$$ \h{2}\big(\cE(-2h)\big) = \h{1}\big(\cE^\vee(2h+K_X)\big) = \h{1}\big(\cE(-2h)\big) ~~ {\rm and} $$
	$$ \h{i}\big(\cE(-3h)\big) = \h{3-i}\big(\cE^\vee(3h+K_X)\big) = \h{3-i}\big(\cE(-h)\big) ~~ {\rm for} ~~i=2,3. $$
	Therefore, the $h$-instanton conditions are reduced to checking the vanishing of $\h0\big(\cE(-h)\big)$ and $\h{1}\big(\cE(-2h)\big)$.
\end{proof}

Finally, the following definition is a generalization, up to a twist by $\cO_X(q_X-2)$, of \cite[Definition 12]{CJ} to the present context.

\begin{definition}\label{defn: rank 0 instanton}
	A \emph{rank 0 $h$-instanton sheaf} on a threefold $(X,h)$ is a 1-dimensional sheaf $\cT$ satisfying $h^{p}(\cT(-2h))=0$ for all $p$. In addition, $d_\cT:=\chi(\cT(-h))$ is called the \textit{degree} of $\cT$.
\end{definition}

Indeed, note that $\cU$ is a rank 0 instanton sheaf on a Fano threefold of Picard rank 1 in the sense of \cite[Defintion 6]{CJ} (that is, $\cU$ is a 1-dimensional sheaf such that $h^{p}(\cU(-q_X\cdot h))=0$) if and only if $\cT=\cU((2-q_X)h)$ is a rank 0 $h$-instanton sheaf.

It is important to observe that rank 0 $h$-instanton sheaves are pure: if $\cZ$ is a non-trivial, 0-dimensional subsheaf of a 1-dimensional sheaf $\cT$, then $\h0(\cT(-2h))\ne0$. Moreover, since $\chi(\cT(th))$ is a degree 1 polynomial on $t$ and $\chi(\cT(-2h))=0$, we must have that $\chi(\cT(th))=d_\cT\cdot(t+2)$.  

\bigskip

With previous definitions in mind, let us establish some key facts about $h$-instanton sheaves that will be useful later on in this paper.

\begin{proposition}
	If $\cE$ is a reflexive, orientable, $h$-instanton sheaf of rank $2$, then $\cE$ is locally free. 
\end{proposition}
\begin{proof}
	If $\cE$ is an orientable, $h$-instanton sheaf of rank $r$, then $\h{p}(\cE(-2))=0$ for all $p\in\bZ$ by definition. By applying the Grothendieck–Riemann-Roch theorem, we obtain 
	$$ \chi(\cE(-2h))=\frac{c_3(\cE)}{2} + \frac{r(r-1)(r-2) K_X^3}{48} + \frac{(2-r)K_X \cdot c_2\big(\cE(-2h) \big)}{4}=0. $$
	Setting $r=2$, we have  $c_3(\cE(-2h))=0$. Therefore, $\cE$ is locally free by \cite[Proposition 2.6]{Har80}.
\end{proof}

Recall that the \textit{homological dimension} of a coherent sheaf $\cE$ is defined by
$$ {\rm dh}(\cE) := \max \{p ~|~ \inext^p(\cE,\cO_X)\ne0 \} ; $$
note that $\cE$ is locally free if and only if ${\rm dh}(\cE)=0$.

\begin{lemma}\label{lem:hd=1}
	The homological dimension of a non-locally free $h$-instanton sheaf is equal to 1.
\end{lemma}
\begin{proof}
	If $\cE$ is a torsion free sheaf then $\inext^2(\cE,\omega_X(th))$ is a 0-dimensional sheaf for any $t\in\bZ$, if non-trivial. Recall that there is $t_0$ such that 
	$$ H^0(\inext^2(\cE,\omega_X(th))) = \Ext^2(\cE,\omega_X(th)) ~~,~~ \forall t>t_0. $$
	By Serre duality, we have that for every $t\in\bZ$
	$$ \Ext^2(\cE,\omega_X(th)) \simeq \Ext^1(\omega_X(th),\cE\otimes\omega_X)^* = \H1(\cE(-th))^*=0 ~~,~~ \forall t>1. $$
	It follows that $\inext^2(\cE,\cO_X)=0$, as desired.
\end{proof}

Let us now turn our attention to rank 0 $h$-instanton sheaves.

\begin{lemma}\label{lem:coho-t}
	If $\cT$ is a rank 0 $h$-instanton sheaf, then $\h0(\cT((-2-n) h))=0$ and $\h1(\cT((-2+n) h))=0$ for every $n\ge0$. In particular, $d_\cT=\h0(\cT(-h))=\h1(\cT(-3h))$.
\end{lemma}

\begin{proof}
	Given a rank 0 $h$-instanton sheaf $T$, let $S\subset X$ be a hyperplane section transversal to the support of $\cT$ (i.e. $\dim(\supp(\cT)\cap S)=0$), so that $\sTor^1(\cT,\cO_S)=0$. This implies that we can twist the exact sequence $0\to\cO_X(-h)\to\cO_X\to\cO_S\to0$ by $\cT(kh)$ to obtain the short exact sequence
	$$ 0 \longrightarrow \cT((k-1)h) \longrightarrow \cT(kh) \longrightarrow \cT\otimes\cO_S(k) \longrightarrow 0. $$
	Taking cohomology, we conclude that $h^0(T((k-1)h))=0$ whenever $h^0(\cT(kh))=0$, while $h^1(\cT(kh))=0$ whenever $h^1(\cT((k-1)h))=0$, since $\dim(\cT\otimes\cO_S)=0$. The desired claim follows by induction on $k$. In particular, we obtain that $\h1(\cT(-h))=0$, thus $\h0(\cT(-h))=\chi(\cT(-h))=d_\cT$. Similarly, $\h0(\cT(-3h))=0$, thus $\h1(\cT(-3h))=-\chi(\cT(-3h))=d_\cT$.
\end{proof}

\begin{proposition} \label{prop:ext2_dual}
	If $\cT$ is a rank 0 $h$-instanton sheaf on $(X, h)$ of degree $d_{\cT}$, then so is $\cT^{\rm D}(4h)$, where $\cT^{\rm D}:=\inext^2(\cT,\omega_X)$. In addition, $\inext^3(\cT,\cO_X)=0$.
\end{proposition}

\begin{proof}
	Firstly, $\cT^{\rm D}$ is a pure 1-dimensional sheaf and $\inext^q (\cT,\omega_X) = 0$ for $q =0,1$ by \cite[Proposition 1.1.6]{HL10}. Using the spectral sequence 
	\[ E^{p,q}_2=\H {p} (\inext^q(\cT, \omega_X (th)))\Rightarrow \Ext^{p+q}(\cT, \omega_X (th)), \]
	and vanishing of $\inext^q (\cT,\omega_X)$ for $q =0,1$, we have two isomorphisms:
	$$  \H {0} (\inext^2(\cT, \omega_X (th))) \simeq \Ext^{2}(\cT, \omega_X (th)) ~~{\rm and} $$
	$$ \H {1} (\inext^2(\cT, \omega_X (th))) \oplus \H {0} (\inext^3(\cT, \omega_X (th))) \simeq \Ext^{3}(\cT, \omega_X (th)). $$
	
	
	Serre duality yields, for $p=0,1$ 
	\[ \Ext^{p+2}(\cT, \omega_X (th)) \simeq \Ext^{1-p}(\omega_X (th), \cT \otimes \omega_X)^{*} \simeq \H{1-p} (\cT(-th))^{*}. \]
	Therefore, setting $\cU:=\cT^{\rm D}(4h)$, we obtain:
	$$ \H{0}(\cU(-2h)) = \H {0} (\inext^2(\cT, \omega_X (2h))) \simeq \H{1} (\cT(-2h))^{*}=0, ~~{\rm and} $$
	$$ \H{1}(\cU(-2h)) = \H {1} (\inext^2(\cT, \omega_X (2h))) \hookrightarrow \H{0} (\cT(-2h))^{*}=0, $$
	so $\cU$ is a rank 0 $h$-instanton sheaf. In addition,
	$$ d_{\cU}=\h0(\cU(-h))=\h1(\cT(-3h))=d_{\cT}. $$
	
	Finally, we also have that
	$$ \H{1}(\cU(-3h)) = \H {1} (\inext^2(\cT, \omega_X (h))) \hookrightarrow \H{0} (\cT(-h))^{*}; $$
	since $\h1(\cU(-3h))=d_{\cU}=d_{\cT}=\h{0}(\cT(-h))$, the map above must be an isomorphism, thus 
	$\H {0} (\inext^3(\cT, \omega_X (h)))=0$. Since $\dim\inext^3(\cT, \omega_X (h))=0$, we conclude that $\inext^3(\cT,\cO_X)=0$.
\end{proof}

We close this section with a simple example of rank 0 $h$-instanton sheaves on threefolds.

\begin{lemma}\label{cor: O_L(1) is rank 0 instanton}
	If $\imath:L\hookrightarrow X$ is a line in a polarized threefold $(X,h)$ such that $\ell\cdot h=1$, where $\ell$ denotes the class of $L$ in $A^2(X)$, then sheaf $\imath_*\big(\cO_{L}(1)\big)$ is a rank 0 $h$-instanton sheaf of degree 1.
\end{lemma}

\begin{proof}
	It is clear that $\imath_*\big(\cO_{L}(1)\big)$ is a pure 1-dimensional sheaf. Also,
	$$ \h{i} \Big(\imath_*\big(\cO_{L}(1)\big) \otimes \cO_{X_c}(-2h) \Big) =
	\h{i}\Big(\imath_*\big(\cO_{L}(-1)\big)\Big) =  \h{i}(L,\cO_{L}(-1)) = 0 $$
	for all $i$. Then $\imath_*\big(\cO_{L}(1)\big)$ is a rank 0 $h$-instanton sheaf on $(X,h)$ of degree \linebreak $d:=\h{0}\Big(\imath_*\big(\cO_{L}(1)\big) \otimes \cO_{X_c}(-h) \Big) = \h{0}(L,\cO_{L})=1$.
\end{proof}


\section{Elementary transformations and classification of rank 2 $h$-instantons} \label{sec:elem transf}

Let $\cE$ be an $h$-instanton sheaf with defect $0$ and $\cT$ be a rank 0 $h$-instanton sheaf on a projective threefold $(X, h)$. Given an epimorphism $\varphi \colon \cE \to \cT$, the sheaf $\cE_{\varphi}:= \ker(\varphi)$ is called an \textit{elementary transformation of $\cE$ along $\cT$}.

\begin{lemma}\label{lem:kernel is instanton sheaf}
	Under the conditions of the previous paragraph, $\cE_{\varphi}$ is an $h$-instanton sheaf with defect 0 of charge $k(\cE)+d_\cT$; in addition, $\cE_{\varphi}$ is orientable if and only if $\cE$ is.
\end{lemma}
\begin{proof}
	Just consider the short exact sequence
	$$ 0 \rightarrow \cE_{\varphi} \rightarrow \cE \rightarrow \cT \rightarrow 0, $$
	and the appropriate exact sequences in cohomology. Note that $c_1(\cE_\varphi)=c_1(\cE)$ and $\chi(\cE_\varphi(-h))=\chi(\cE(-h))-\chi(\cT(-h))=-k(\cE)-d_\cT$.
\end{proof}

In other words, elementary transformations are a way to construct $h$-instanton sheaves of higher charge. However, the modified instanton sheaf $\cE_\varphi$ is never locally free; so to construct locally free $h$-instanton sheaves of arbitrary charge one must show that $\cE_\varphi$ can be deformed into a locally free sheaf. 

Our next result, which generalizes \cite[Theorem 24]{CJ} to the context of the present paper, shows that every rank 2 orientable $h$-instanton sheaf with defect $0$ can be obtained as an elementary transformation of a locally free $h$-instanton sheaf.

\begin{theorem}\label{thm:double_dual_bundle}
	Let $\cE$ be a rank 2 orientable $h$-instanton sheaf with defect $0$ on a smooth irreducible projective variety $(X, h)$ of dimension three; assume that $\cE$ is not locally free. Then $\cE^{\vee\vee}$ is an orientable locally free $h$-instanton sheaf and $\cT_E:=\cE^{\vee\vee}/\cE$ is a rank $0$ $h$-instanton sheaf.
\end{theorem}

\begin{proof}
	We have the following standard short exact sequence of sheaves on $X$ 
	\begin{equation}\label{ses:double_dual}
		0 \rightarrow \cE \rightarrow \cE^{\vee\vee} \rightarrow \cT_\cE \rightarrow 0.
	\end{equation}
	If we tensor the sequence \eqref{ses:double_dual} by $\cO_X(th)$ then the cohomology sequence gives us the inequality
	$$ \h i (\cE^{\vee\vee} (th)) \leq \h i (\cE(th)) + \h i (\cT_\cE (th)) \ \ \ \text{for} \ \ i>0. $$
	Since $\cE$ is an instanton sheaf, $\h2 (\cE(-2h)) = \h3 (\cE(-3h))= 0$. Also, $\h2 (\cT_\cE (-2h)) = \h3 (\cT_\cE (-3h)) =  0$, since $\cT_\cE$ is a 1-dimensional sheaf. Therefore, $\h2 ({\cE^{\vee\vee} (-2h)}) = \h3 ({\cE^{\vee\vee} (-3h)}) = 0$.
	
	Next, notice that $\h0 (\cE^{\vee\vee} (-h)) = \h3 (\cE^{\vee\vee\vee} (h) \otimes \omega_X)$ by Proposition \cite[Theorem 2.5]{Har80}. Then, since $\cE^{\vee\vee}$ is a rank 2 reflexive sheaf, we have $\cE^{\vee\vee\vee} \simeq \cE^{\vee\vee} \otimes (\det(\cE^{\vee\vee}))^{-1}$ by \cite[Proposition 1.10]{Har80}. So, $\h0 (\cE^{\vee\vee} (-h)) = \h3 (\cE^{\vee\vee} (-3h)) = 0$.
	
	It remains to prove that $\h1 (\cE^{\vee\vee} (-2)) = 0$ to show that $\cE^{\vee\vee}$ is an $h$-instanton sheaf. For this purpose, let us consider the spectral sequence 
	$$
	\H {p} (\inext^q(\cE^{\vee\vee} , \cO_X(2h + K_X))) \Rightarrow \Ext^{p+q}(\cE^{\vee\vee} , \cO_X(2h + K_X)). 
	$$
	Since $\cE^{\vee\vee}$ is reflexive, we know that $\H {p} (\inext^q(\cE^{\vee\vee} , \cO_X(2h + K_X))) = 0 $ when $q>1$ and when $q=1$ and $p>0$ by  \cite[Proposition 1.1.10]{HL10}. Also, for $q=0$, we have
	$$\H {p} (\inext^0(\cE^{\vee\vee} , \cO_X(2h + K_X))) \simeq \H p (\cE^{\vee\vee\vee} (2h + K_X)) \simeq \H p (\cE^{\vee\vee} (-2h)).	$$
	So, $\H {p} (\inext^q(\cE^{\vee\vee} , \cO_X(2h + K_X))) $ vanishes also for $q = 0$ and $p = 2,3$. Therefore, the spectral sequence gives us the isomorphism 
	\begin{equation}\label{isomorphism from spectral}
		\Ext^{1}(\cE^{\vee\vee} , \cO_X(2h + K_X)) \simeq \H {0} (\inext^1(\cE^{\vee\vee} , \cO_X(2h + K_X))) \oplus \H1 (\cE^{\vee\vee} (-2h)).	
	\end{equation}
	Using Serre duality for the left-hand side of the isomorphism, we have
	\begin{align*}
		\Ext^{1}(\cE^{\vee\vee} , \cO_X(2h + K_X)) &\simeq \Ext^{2}( \cO_X(2h + K_X), \cE^{\vee\vee} (K_X))^{\vee} \\
		&\simeq \H2(\cE^{\vee\vee}(-2))^{\vee} = 0.
	\end{align*}
	So, both $\H{0}(\inext^1(\cE^{\vee\vee}, \cO_X(2h + K_X)))$ and $ \H1 (\cE^{\vee\vee} (-2h))$ vanish; the latter vanishing proves that $\cE^{\vee\vee}$ is an instanton sheaf, while the former vanishing proves that the singular locus of $\cE^{\vee\vee}$ is empty, which means that $\cE^{\vee\vee}$ is locally free.
	
	Finally, we argue that $\cT_\cE$ is a rank 0 instanton sheaf. The cohomology sequence of \eqref{ses:double_dual} implies the inequality
	\[ \h i (\cT_\cE (-2h)) \leq \h i (\cE^{\vee\vee} (-2h)) + \h {i+1} (\cE(-2h))  \ \ \ \text{for} \ \ i=0,1. \]
	Since we have the vanishing of the right-hand side of the inequality, $\cT_\cE$ is a rank 0 instanton sheaf.
\end{proof}

\begin{remark}
	Recall from \cite[Theorem 1.4]{AC23} that a torsion-free sheaf $\cE$ on $(X,h)$ is an $h$-instanton sheaf $\cE$ with defect $1$ if $\h0 (\cE(-h)) = \h3(\cE(-2h))=\h1(\cE(-2h))=\h2(\cE(-h)) =0$,  $\h1(\cE(-h))=\h2(\cE(-2h))=k$ and $\chi(\cE)=\chi(\cE(-3h))$.
	We note that the result in Theorem \ref{thm:double_dual_bundle} does not hold for a rank 2 orientable $h$-instanton sheaf with \textit{defect $1$}. Because $\H {2} (\inext^0(\cE^{\vee\vee} , \cO_X(h + K_X))) \simeq \h2 (\cE^{\vee\vee} (-2h))$ does not vanish in principle in the case of defect $1$. Due to this reason, the spectral sequence 
	$$\H {p} (\inext^q(\cE^{\vee\vee} , \cO_X(h + K_X))) \Rightarrow \Ext^{p+q}(\cE^{\vee\vee} , \cO_X(h + K_X))$$
	does not stabilize on the second page. Therefore we will not have an isomorphism similar to one in display \eqref{isomorphism from spectral}, which is necessary for the vanishing of $\h1(\cE^{\vee\vee}(-2h)).$
\end{remark}



\section{Ruled Fano threefolds of index 1 and Picard rank 2} \label{sec:ruled fanos}

This section will revise the relevant facts about ruled Fano threefolds of index 1 and Picard rank 2.

As we mentioned in the Introduction, these varieties occur in 5 deformation classes given by the projectivization of the rank 2 bundles $F_c$ over $\p2$ for $0 \leq c \leq 4$, listed in Table \ref{table: for F}. We set $X_c:=\mathbb{P}_{\p2}\big(F_c\big)$ and let $\pi:X_c\to\p2$ the the projection induced from $F_c$.

It is known that the Chow ring of $X_c$ is given by
\begin{equation}\label{eqn:chow_ring_projectivization}
	A(X_c)=A\left(\p2 \right)[\xi] / \left( \xi^{r}  + c_{1}(F_c^{\vee}) \xi^{r-1} +c_{2} (F_c^{\vee})\right)
\end{equation}
where $\xi = c_1(\cO_{X_{c}} (1))$, see \cite[Theorem 9.6]{EH13}; we remark that $\xi$ restricts to the hyperplane class on each fiber.

We have two short exact sequences
\begin{equation}\label{ses:tangentbundle_1}
	0 \rightarrow T_{X_c \mid \p2} \rightarrow T_{X_c} \rightarrow \pi^{*} T_{\p2} \rightarrow 0
\end{equation}
\begin{equation}\label{ses:tangentbundle_2}
	0 \rightarrow \mathcal{O}_{X_c} \rightarrow \pi^{*} F_c^{\vee} \otimes \cO_{X_c}(\xi) \rightarrow T_{X_c \mid \p2} \rightarrow 0
\end{equation}
where $T_{X_c \mid \p2}$ is the relative tangent sheaf and $\pi : X_c \rightarrow \p2$. One can compute the Chern classes of $T_{X_c}$ using short exact sequences \eqref{ses:tangentbundle_1} and \eqref{ses:tangentbundle_2}. Namely,
\begin{equation}\label{eqn:1_chern_class_tangent}
	c_1(T_{X_c}) = c_1(\pi^{*} T_{\p2}) + c_1 \left( \pi^{*} F_c^{\vee} \otimes \cO_{X_c}(\xi) \right)
\end{equation}
\begin{equation}\label{eqn:2_chern_class_tangent}
	c_2(T_{X_c}) = c_2 \left( \pi^{*} F_c^{\vee} \otimes \cO_{X_c}(\xi) \right) + c_2(\pi^{*} T_{\p2}) + c_1 \left( \pi^{*} F_c^{\vee} \otimes \cO_{X_c}(\xi) \right) c_1(\pi^{*} T_{\p2})
\end{equation}

Set $f:=c_1(\pi^*(\cO_{\p2}(1)))$; then, by the description in display \eqref{eqn:chow_ring_projectivization}, 
$\operatorname{Pic}(X_c)=\langle \xi , f \rangle$ with relations $\xi^{2} = 2 \xi f - c  f^2 $ and $ f^{3}=0$. Also, using the equations in \eqref{eqn:1_chern_class_tangent} and in \eqref{eqn:2_chern_class_tangent}, we have
$$ c_1(T_{X_c}) = 2\xi + f ~~ {\rm and} ~~ c_2 (T_{X_c}) = 6 \xi f  - 3f^2. $$
Since $\pi_{*} (\xi f^2) = f^2$ by \cite[Lemma 9.7]{EH13},
$\operatorname{deg}(\xi f^2) = 1$. Therefore, we have
$$ \operatorname{deg}(\xi^3) = 4-c ~~,~~ \operatorname{deg}(\xi^2 f) = 2 ~~{\rm and}~~ \operatorname{deg}(f^3) = 0. $$

We also have the following short exact sequences:
\begin{equation}\label{ses:general_ses_for_our_variety_1}
	0 \rightarrow \pi^{*} \Omega_{\mathbb{P}^{2}} \rightarrow \mathcal{O}_X (-f)^{\oplus3} \rightarrow \mathcal{O}_{X_c} \rightarrow 0
\end{equation}
\begin{equation}\label{ses:general_ses_for_our_variety_2}
	0 \rightarrow \mathcal{O}_{X_c} (-3 f) \rightarrow \mathcal{O}_{X_c}(-2 f)^{\oplus3} \rightarrow \pi^{*} \Omega_{\mathbb{P}^{2}} \rightarrow 0
\end{equation}
\begin{equation}\label{ses:general_ses_for_our_variety_3}
	0 \rightarrow \cO_{X_c}(2f - \xi) \rightarrow \pi^{*} F_c  \rightarrow \mathcal{O}_{X_c}(\xi) \rightarrow 0
\end{equation}


The sequence in display \eqref{ses:general_ses_for_our_variety_1} is the pull-back of dual Euler sequence on $\p2$; the one in display \eqref{ses:general_ses_for_our_variety_2} is the dual of \eqref{ses:general_ses_for_our_variety_1} twisted by $\cO_{X_c}(-3 f)$; finally, the sequence in display \eqref{ses:general_ses_for_our_variety_3} is the dual of sequence \eqref{ses:tangentbundle_2} twisted by $\cO_{X_c}(\xi)$, by noticing that $T_{X_c \mid \p2} \simeq \cO_{X_c}(2\xi - 2f).$ 

Since, for each $c$, $F_c$ is globally generated, $\cO_{X_{c}}(\xi)$ is globally generated. Similarly, $\cO_{X_{c}}(f)$ is globally generated. Also, $\cO_{X_{c}}(h) \simeq \cO_{X_{c}}(\xi + f)$ is ample by \cite[Proposition 2.1]{SW90}. In this section, we mean by $(X_c, h)$ the polarized variety $X_c =\mathbb{P}_{\p2}(F_c)$ with the ample bundle $h = \xi + f$.

\begin{lemma}\label{lem:cohomology_of_line_bundles}
	Let $\mathcal{L} := \mathcal{O}_{X_c} \left( \lambda_{1} \xi+ \lambda_{2} f \right) $ be a line bundle on $X_c$. Then
	\[
	\H{i} (X_c, \mathcal{L} ) \simeq \begin{cases} 
		H^{i} \left( \mathbb{P}^{2}, S^{\lambda_{1}}(F_c) \otimes \mathcal{O}_{\mathbb{P}^{2}} \left( \lambda_{2} \right) \right) & \mbox{if } \lambda_1 \geq 0 \\
		0 & \mbox{if } \lambda_1 = -1 \\
		H^{i-1} \left( \mathbb{P}^{2},  S^{-\lambda_{1}-2}(F_c)^{\vee} \otimes \mathcal{O}_{\p2} \left( \lambda_{2} -2 \right) \right) & \mbox{if } \lambda_1 \leq -2.
	\end{cases}
	\]
	where $S^m F_c$ stands for the m-th symmetric power of $F_c$ for any $m \geq 0$.
\end{lemma}
\begin{proof}
	Let us use Leray spectral sequence
	$$\H{i} \left(\p2,  R^{j} \pi_{*}\mathcal{L} \right) \Longrightarrow \H{i+j} (X,\mathcal{L}), $$
	and note that $\mathcal{L}=\mathcal{O}_{X_c}\big(\lambda_{1}\xi\big)\otimes\pi^*(\cO_{\p2}(\lambda_2))$. We analyze each of the three cases, using \cite[Exercise 8.4, Chapter III]{Har77}:
	\begin{itemize}
		\item if $\lambda_{1} \geq 0$, then $\pi_{*} \left( \mathcal{O}_{X_c} \left(\lambda_{1} \xi \right) \right) = S^{\lambda_{1}} F$ and $R^{1} \pi_{*} \left( \mathcal{O}_{X_c} \left(\lambda_{1} \xi \right) \right) = 0$. The desired isomorphism follows from the projection formula and the Leray spectral sequence. \medskip
		
		\item when $\lambda_{1} = -1$ we have that $\pi_{*} \left( \mathcal{O}_{X_c} \left( - \xi \right) \right) = R^{1} \pi_{*} \left( \mathcal{O}_{X_c} \left( - \xi \right) \right) = 0$, so the Leray spectral sequence implies that $\H{i}(X_c,\mathcal{L})=0$ for all $i\ge0$.  \medskip
		
		\item if $\lambda_{1} \leq -2$, then $\pi_{*}\left(\mathcal{O}_{X_c} \left(\lambda_{1} \xi \right)\right)=0$ and
		$$ R^{1}\pi_{*}\left(\mathcal{O}_{X_c} \left( \lambda_{1} \right) \right)  =\pi_{*} \left( \mathcal{O}_{X_c} \left( -2 - \lambda_{1} \right) \right)^{\vee} \otimes \left( {\wedge }^{2}  F_c \right)^{\vee}  =S^{-\lambda_{1}-2}(F_c)^{\vee} \otimes \cO_{\p2} (-2f).$$
		Again, the desired isomorphism follows from the projection formula and the Leray spectral sequence.
	\end{itemize}
\end{proof}

The following technical result will be useful in the next section.

\begin{proposition}\label{prop:non effective divisors}
	If $a,b\in\bZ$ satisfy $a\geq0$ and $(9-c)a+4b\leq c-15$ with $c\in\{0,\dots,4\}$, then
	\begin{enumerate}
		\item $a\xi + (b+2)f$ is not effective for each $c$.
		\item $(a+2)\xi + (b+1)f$ is not effective for $c = 2,3,4$.
	\end{enumerate}
\end{proposition}

\begin{proof}
	First, notice that $\xi$ and $f$ are globally generated divisors since both $F_c$ and $\cO_{\p2}(1)$ are globally generated bundles. Then, $k\xi+f$ is a globally generated divisor for any $k>0$ as the sum of globally generated divisors. Then, for any irreducible curve $C \subseteq X_c$, $(k\xi+f) \cdot C = (\xi +f) \cdot C + (k-1)\xi \cdot C$. Since $(\xi +f)$ is ample, $(\xi +f) \cdot C > 0$ by Nakai--Moishezon Criterion. Since $(k-1)\xi$ is globally generated, $(k-1)\xi \cdot C \geq 0$. So, $k\xi+f) \cdot C > 0$, and so, $(k\xi+f)$ is ample by \cite[Corollary 1.2.15]{Laz17}.
	
	Now consider the intersection product of $a\xi+(b+2)f$ with $(k \xi + f)^2$, which is $((4-c)a +2b +4)k^2 + (4a+2b+4)k +a$. The hypotheses imply that 
	$$ ((4-c)a +2b +4) \leq \frac{c-15}{2} - \frac{1+c}{2}a +4 \leq \frac{c-7}{2} $$ 
	Since $c-7<0$, the Nakai--Moishezon Criterion implies that $a\xi + (b+2)f$ cannot be effective.
	
	Next, consider the intersection product of $(a+2) \xi + (b+1)f$ with $(k \xi + f)^2$, which is $((4-c)a +2b + 10 -2c)k^2 + (4a+2b+10)k +a+2$. Then we have, using the hypotheses on $a$ and $b$:
	$$ ((4-c)a +2b +10-2c) \leq \frac{c-15}{2} - \frac{1+c}{2}a + 10-2c \leq \frac{5-3c}{2}. $$ 
	Since, $5-3c< 0$ for $c=2,3,4$, the same argument applies and $(a+2) \xi + (b+1)f$ cannot be effective.
\end{proof}

To complete this section, let $M$ be a curve in the class of $\xi f$ in $A^2(X_c)$; since $M$ is a complete intersection, 
its structure sheaf admits the following 
resolution:
\begin{equation}\label{ses: resolution of elliptic}
	0 \lra \cO_X (-\xi - f) \lra \cO_X(-\xi) \oplus \cO_X(-f) \lra \cO_X \lra \cO_M \lra 0.
\end{equation}
In particular, we have $\mathcal{N}_{M|X} \simeq \cO_M (1) \oplus  \cO_M (2)$, thus $\det(\mathcal{N}_{M|X})=\cO_M (3)$. 

Moreover, $\xi f (\xi +f) = 3$ and 
$$ \chi(\cO_M) = \chi(\cO_X) - \chi(\cO_X(-f)) - \chi(\cO_X(-\xi)) + \chi(\cO_X (-\xi - f)) = 1 $$
since $\chi(\cO_X(-f))=\chi(\cO_X(-\xi))=\chi(\cO_X(-\xi - f))=0$ by Lemma \ref{lem:cohomology_of_line_bundles}. Therefore, $M$ is a rational curve of degree 3 in $X_c$.

In what follows we will denote by $\Lambda_M$ the component of the Hilbert scheme ${\rm Hilb}^{3t+1}(X_c)$ containing the complete intersection curves described above. Since
$$ \h0(\mathcal{N}_{M|X})=5 ~~ {\rm and} ~~ \h1(\mathcal{N}_{M|X})=0, $$
we conclude that $\Lambda_M$ is smooth of dimension equal to 5. Note that generic elements of $\Lambda_M$ do not intersect each other. If they intersect, two generic elements $\xi_0 f_0$ and $\xi_1 f_1$ of $\Lambda_M$ have intersection point $p$. Notice that, the divisor $f_1$ will intersect $\xi_0 f_0$ at $p$ and then $\xi_1$ will go through $p$. Then any element $\xi_i f_1$ of $\Lambda_M$ will intersect  $\xi_0 f_0$ at $p$; that is, $\xi_i$ goes through $p$ for any $i$. This is a contradiction because $\cO_{X_c}(\xi) $ is globally generated.

\begin{lemma} \label{lem:serreconstruction}
	For each $c=0,1,2,3,4$ and $m\ge2$, there exists a pair $(\mathcal{F},s)$ consisting of rank 2 locally free sheaf $\mathcal{F}$ on $X_c$ and a global section $s\in \H0(\mathcal{F})$ vanishing exactly on a disjoint union of $m$ curves in $\Lambda_M$.
\end{lemma}
\begin{proof}
	We use the Serre construction, see \cite{Ar07} for details. More precisely, consider the 1-dimensional subscheme $Z\subset X_c$
	\begin{equation}\label{XInstanton}
		Z:=\bigcup_{i=1}^{m}M_i  ~~,~~ {\rm where}~~ M_{i}\in \Lambda_M ~~ {\rm and}~~ M_i\cap M_j=\emptyset ~~ {\rm when}~~ i\ne j,
	\end{equation}
	and extensions of the form
	\begin{equation} \label{eq:bdl-f}
		0 \lra \cO_{X_c} \lra \cF \lra \cI_{Z}(2\xi-f) \lra 0.
	\end{equation}
	Using Serre duality we obtain
	\begin{align*}
		\Ext^1(\cI_{Z}(2\xi-f),\cO_{X_c}) & \simeq \H2(\cI_{Z}(-2f))^* \simeq \bigoplus_{i=1}^m \H1(\cO_{M_i}\otimes\cO_{X_c}(-2f))^* \\
		&  \simeq \bigoplus_{i=1}^m \H1(M_i,\cO_{M_i}(-2))^*  \simeq \bigoplus_{i=1}^m \H0(M_i,\cO_{M_i}).
	\end{align*}
	Therefore $\dim\Ext^1(\cI_{Z}(2\xi-f),\cO_{X_c})=m$, and the existence of locally free extensions is guaranteed by the existence of non-vanishing sections in $\bigoplus_{i=1}^m H^0(M_i,\cO_{M_i})$.
	
	When $c\ge1$, a general curve $Z$ of class $|m\cdot\xi f|$ is not contained in a divisor of class $|2\xi|$ when $m\ge 3$; therefore, we can assume that $\h0(\cI_{Z}(2\xi))=0$. When $c=0$, then $2\xi f=\xi^2$, so we need $Z$ to have at least 5 connected components to make sure that $\h0(\cI_{Z}(2\xi))=0$.
	
	If $\h0(\cI_{Z}(2\xi))=0$, then $\h0(\cI_{Z}(2\xi-f))=0$ and $h^0(\cF)=1$. It follows that the family of sheaves constructed as above has dimension equal to
	\begin{equation} \label{eq:dim.family}
		\dim(\Lambda_M)m+(m-1) = 6m-1 ;
	\end{equation}
	the first summand corresponds to the choice of $m$ disjoint curves in $\Lambda_M$, while the second summand corresponds to $\dim\Ext^1(\cI_{Z}(2\xi-f),\cO_{X_c})-\h0(\cF)$.
\end{proof}

Finally, we will test $\mu$-stability of rank 2 bundles on $(X_c, h)$ using the following criterion.

\begin{proposition}\label{prop:Hoppe}
	Let $\cE$ be a rank 2 orientable locally free sheaf on $(X_c, h)$. Then $\cE$ is $\mu$-(semi)stable if and only if $\H0(\cE(a\xi + bf)) =0 $ for all $a,b \in \mathbb{Z}$ satisfying $ (9-c) a + 4b \leq ~(<)~ c-15$. 
\end{proposition}

\begin{proof}
	First, notice that the slope $\mu (\cE)$ of $\cE$ is $\frac{c_1 (\cE) \cdot h^2}{2} = 15-c$ and $(a \xi + bf) \cdot h^2 =  (9-c) a + 4b$. The result follows from \cite[Corollary 4]{JMPS17}.
\end{proof}


\section{Rank 2 $h$-instanton bundles on $X_c$} \label{sec:instantons on Xc}


We start by noticing that if $\cE$ is orientable rank 2 bundle on $(X_c,h)$, then $c_1(\cE)=2 \xi + 3f$. The instanton condition imposes restrictions on $c_2(\cE)$.

\begin{lemma}\label{lem:bounds_2ndChern}
	Let $\cE$ be an orientable rank 2 $h$-instanton bundle on $(X_c,h)$. If 
	$c_2(\cE) = \alpha \xi f + \beta f^2$, then $k(\cE)=3 \alpha + \beta - 21 + c\ge0$ and $\alpha \geq 5$.
\end{lemma}
\begin{proof}
	Let $\cE$ be a rank 2 vector bundle on $X_c$ with $c_1(\cE) = 2 \xi + 3f$ and $c_2(\cE) = \alpha \xi f + \beta f^2$. A lengthy but straightforward computation with the Riemann--Roch theorem yields
	\begin{align}
		\chi (\cE (\lambda_1 \xi + \lambda_{2} f)) &= \bigg( \frac{4-c}{3} \bigg) \lambda_1^3 + 2 \lambda_{1}^2 \lambda_{2} + \lambda_{1} \lambda_{2}^2 +(12-2c) \lambda_{1}^2 + 12 \lambda_{1} \lambda_{2} +2 \lambda_{2}^2 \nonumber\\
		&+ \bigg( \frac{143-14c}{3} - 2 \alpha- \beta \bigg) \lambda_{1} + \big( 21- \alpha \big) \lambda_{2}  -6 \alpha -2 \beta +68-4c. \label{eq:RR}
	\end{align}
	Setting $\lambda_1=\lambda_2=-1$, we get that 
	$$ \chi (\cE (-h)) = -3\alpha - \beta +21 - c, $$
	leading to the desired formula for the charge of $\cE$.
	
	
	Since $\xi$ is an effective divisor and $\h0 (\cE (-h)) = 0$ by definition, we have the vanishing $\h0 (\cE (-2 \xi - f)) = 0$. Also, 
	$$ \h3(\cE (-2 \xi - f))=\h0 (\cE(-2\xi-3f)) $$
	by Serre duality and the isomorphism $\cE^\vee\simeq\cE(-2\xi-3f)$; the term on the right-hand side vanishes again because $2\xi+3f$ is an effective divisor, so $\h3(\cE (-2 \xi - f))=0$.
	
	If we tensor the sequence \eqref{ses:general_ses_for_our_variety_2} with $\cE(-2 \xi -f)$ and pass to cohomology, we obtain the inequality 
	$$ \h3 \big( \cE(-2 \xi - f) \otimes  \pi^{*} \Omega_{\mathbb{P}^{2}} \big) \leq \h3 (\cE (-2\xi - 3f)), $$
	thus $\h3 \big( \cE(-2 \xi - f) \otimes  \pi^{*} \Omega_{\mathbb{P}^{2}} \big)=0$.
	
	Similarly, if we tensor the sequence \eqref{ses:general_ses_for_our_variety_1} with $\cE(-2 \xi -f)$ and pass to cohomology, we obtain
	$$ \h2 (\cE (-2 \xi - f)) \leq \h2 (\cE (-2 \xi - 2f)) + \h3 \big( \cE(-2 \xi - f) \otimes  \pi^{*} \Omega_{\mathbb{P}^{2}} \big) = 0, $$
	so $\h2 (\cE (-2 \xi - f)) = 0$. Therefore, we have $\chi (\cE (-2 \xi - f)) = - \h1 (\cE (-2 \xi - f))$.  Since $\chi (\cE (-2 \xi - f)) = -\alpha +5 $ by the formula in \eqref{eq:RR}, we obtain the second inequality.
\end{proof}



We are now ready to state the main result of this section, establishing the existence of rank 2 instanton bundles on the varieties $X_c$.

\begin{proposition} \label{thm:existence_of_instanton_by_Serre}
	For each $\alpha \ge 6$ when $c\ge1$ and $\alpha\ge7$ when $c=0$, there exist rank 2, orientable, $\mu$-stable $h$-instanton bundles $\cE$ on $(X_c,h)$ with $c_2(\cE) = \alpha \xi f + 2f^2$.
\end{proposition}

\begin{proof}
	Set $\cE:=\cF(2f)$, where $\cF$ is one of the locally free sheaves constructed in Lemma \ref{lem:serreconstruction}. Note that $c_1(\cE) = c_1(\cF) + 4f = 2\xi + 3f$ so $\cE$ is orientable; moreover, 
	$$ c_2 (\cE)= c_2(\cF) + c_1(\cF) \cdot 2f + 4f^2 = (m+4) \xi f + 2f^2. $$ 
	To simplify notation, we set $\alpha=m+4$.
	
	We argue that $\cE$ is an $h$-instanton bundle when $m\ge2$ (equivalently $\alpha\ge6$) if $c\ge1$, and when $m\ge3$ (equivalently $\alpha\ge7$) if $c=0$.
	
	By Lemma \ref{lem:rk2instanton}, it is enough to check that $\h0(\cE(-h)) =\h1(\cE(-2h))=0$. Twisting the exact sequence in display \eqref{eq:bdl-f} by $\cO_{X_c}(2f)$ we obtain
	\begin{equation}\label{ses:constructed_bundle_by_Serre}
		0\longrightarrow\cO_{X_c}(2f)\longrightarrow \cE \longrightarrow \cI_{Z}(2\xi + f) \longrightarrow 0.
	\end{equation}
	Since $\h{i}(\cO_{X_c}(-\xi+f))=0$ by Lemma \ref{lem:cohomology_of_line_bundles}, the cohomology of exact sequence \eqref{ses:constructed_bundle_by_Serre} tensored by $\cO_{X_c}(-h)$ yields $\h0(\cE(-h))=\h0(\cI_{Z}(\xi))$. When $c\ge1$, then $\h0(\cI_{Z}(\xi))=0$ for $m\ge2$, since the disjoint union of two generic curves of class $|\xi f|$ cannot be contained in a single divisor of class $|\xi|$. However, this argument fails when $c=0$ because $\xi^2=2\xi f$, which implies that the union of two curves of class $|\xi f|$ does lie in a divisor of class $|\xi|$, therefore $\h0(\cI_{Z}(\xi))\ne0$ when $m=2$ (equivalently, $\alpha=6$). When $c=0$, we need $Z$ to have at least 3 disjoint curves (so that $\alpha\ge7$) to make sure that $\h0(\cI_{Z}(\xi))=0$.
	
	Lemma \ref{lem:cohomology_of_line_bundles} also yields $\h{i}(\cO_{X_c}(-2\xi ))=0$ for all $i$, so $\h1(\cE(-2h))=\h1(\cI_{Z}(-f)).$
	The cohomology of the following short exact sequence 
	\begin{equation}\label{ses:ideal_sheaf}
		0 \lra \cI_{Z}(-f)  \lra \cO_{X_c}(-f) \lra \cO_{Z}(-f) \lra 0
	\end{equation}
	gives us the equality $\h1 (\cI_{Z} (-f)) = \h0 (\cO_{Z} (-f))$ since $\h{i}(\cO_{X_c}(-f))=0$ by Lemma \ref{lem:cohomology_of_line_bundles}. Since $\mathrm{deg} (\xi f \cdot (-f)) = -1$, we get that $\h0 (\cO_{Z} (-f)) = 0$.
	
	Next, let us prove that $\cE$ is $\mu$-stable; by Proposition \ref{prop:Hoppe}, it is enough to show that $\h0(\cE (a\xi + bf)) = 0$ for all $a,b\in\mathbb{Z}$ satisfying
	\begin{equation} \label{eq:ineq}
		(9-c) a + 4b \leq c-15.
	\end{equation}
	
	If we tensor sequence in display \eqref{ses:constructed_bundle_by_Serre} by $\cO_{X_c}(a\xi + bf)$ then the induced cohomology sequence gives us
	\begin{equation}\label{ineq:existence_theorem_1}
		\h0(\cE(a\xi + bf)) \leq \h0 (\cO_{X_c} (a\xi + (b+2)f)) + \h0 ( \cI_{Z}((a+2)\xi + (b+1)f));
	\end{equation}
	we will now check that both summands vanish when $a$ and $b$ satisfy the inequality in display \eqref{eq:ineq}.
	
	Note that $\h0 (\cO_X(a\xi + (b+2)f)) = 0$ if $a < 0$ by Lemma \ref{lem:cohomology_of_line_bundles}. If $a\geq0$, then Proposition \ref{prop:non effective divisors} guarantees that $a\xi + (b+2)f$ is not effective; therefore, we conclude that $\h0(\cO_X (a\xi + (b+2)f))=0$ for every $a,b$ satisfying $(9-c)a+4b\leq c-15$. 
	
	The cohomology of the sequence in display \eqref{ses:ideal_sheaf} after tensoring by $\cO_X ((a+2)\xi + (b+1)f)$ gives us the inequality
	\begin{equation}\label{ineq:existence_theorem_2}
		\h0 ( \cI_{Z}((a+2)\xi + (b+1)f)) \leq \h0 ( \cO_X ((a+2)\xi + (b+1)f)).
	\end{equation}
	We now need to consider various cases:
	\begin{itemize}
		\item If $a = -2$ then the inequality in display \eqref{eq:ineq} implies that $ b+1 \leq 1$. Since $m\ge2$, $Z$ will not lie in an element of the class $|f|$; that is, $\h0(\cI_{Z}(f))=0$. Since  $f$ is effective, we get that  $\h0 (\cI_{Z} ((b+1)f)) = 0$ for $b+1\leq1$. \\
		
		\item If $a =-1$ then $ b \leq -2$. Since $\h0(\cE(-\xi-f))=0$ by Proposition \ref{lem:cohomology_of_line_bundles} and $f$ is effective, then $\h0(\cE(-\xi+bf)) =0$ for $b\leq-2$. \\
		
		\item  If $a\geq0$ then $\cO_X((a+2)\xi + (b+1)f)$ is not effective for $c=2,3,4$ by Proposition \ref{prop:non effective divisors}. So, $\h0(\cO_X ((a+2)\xi + (b+1)f))=0$.  \\
	\end{itemize}
	
	This completes the proof of $\mu$-stability for $\cE$ when $c=2,3,4$. 
	
	Assume now that $c=1$; the inequality in display \eqref{eq:ineq} simplifies to $8a+4b\le-14$. When $a\geq0$, we have that $a+b\leq\frac{-14-4a}{4} \leq -4$. By Lemma \ref{lem:cohomology_of_line_bundles}, we have 
	\begin{align*}
		\h0 (\cO_{X_1}((a+2)\xi + (b+1)f) &= \h0 (\p2 , S^{a+2}F_1\otimes \cO_{\p2}(b+1)) \\
		& = \h0 (\p2, \cO_{\p2}^{a+3}(a+2) \otimes \cO_{\p2}(b+1)) \\
		& = \h0 (\p2, \cO_{\p2}^{a+3}(a+b+3)) =0.
	\end{align*}
	Therefore, $\cE$ is $\mu$-stable when $c=1$ and $m\ge2$.
	
	Finally, set $c=0$; the inequality in display \eqref{eq:ineq} simplifies to $9a+4b\le-15$ so that $2a+b\le-(15+a)/4$. Again, we must consider three separate cases.
	\begin{itemize}
		\item If $a=0$ then $b\leq-4$. Then, if we tensor the sequence in display \eqref{ses:general_ses_for_our_variety_3} by $\cE(-\xi -2f)$, we have
		$$\h0(\cE(-4f)) \leq \h0(\cE(-\xi-4f))+\h0(\cE(-\xi-2f))+\h1(\cE(-2\xi-2f)).$$
		Since $\cE$ is an instanton, all terms on the right-hand side vanish. Therefore, we have $\h0(X_0, \cE(-4f))=0$, hence $\h0(X_0, \cE(bf))=0$ for $b \leq -4$.\\
		
		\item If $a=1$, then $b\leq-6$. Then, if we tensor the sequence in display \eqref{ses:general_ses_for_our_variety_3} by $\cE(-6f)$, we have $$ \h0(\cE(-\xi-6f)) \leq \h0(\cE(-4f)) + \h0(\cE(-6f)) + \h1(\cE(-\xi-4f)). $$
		Note that $\h0(\cE(-4f))$ and $ \h0(\cE(-6f))$ vanish by the previous item. Also, if we tensor the sequence is display \eqref{ses:constructed_bundle_by_Serre} by $\cO_{X_0}(-\xi-4f)$, we have 
		$$ \h1(\cE(-\xi-4f)) \leq \h1(\cO_{X_0}(-\xi-2f)) + \h1\cI_{Z}(\xi-3f)). $$
		Since $ \h1(\cO_{X_0}(-\xi-2f)) = 0$ by Lemma \ref{lem:cohomology_of_line_bundles}, we have 
		$$ \h1(\cE(-\xi-4f)) \leq \h1(\cI_{Z}(\xi-3f)) \leq \h1(\cO_{X_0}(\xi-3f)) + \h0(\cO_{Z}(\xi-3f)). $$
		Since $(\xi-3f)\cdot \xi f = -1$, $\h0(X_0, \cO_{Z}(\xi-3f)) = 0$. Moreover, $\h1(\cO_{X_0}(\xi-3f))=0$ again by Lemma \ref{lem:cohomology_of_line_bundles}. So, $\h0(\cE(-\xi-6f)) = 0$. Hence, $\h0(\cE(-\xi+bf)) = 0 $ for $b \leq -6$.\\
		
		\item If $2 \leq a \leq 5$, one can use the same arguments as in the case $a=1$. \\
		
		\item If $a\geq6$ then $2a+b\leq -6$. Then we have
		\begin{align*}
			\h0 (X_0, \cE(a\xi + bf)) &\leq \h0 ( \cO_X ((a+2)\xi + (b+1)f))\\
			&\le \h0 (\p2, S^{a+2}F_0 \otimes \cO_{\p2} (b+1)) \\
			&\le \h0 (\p2, \oplus_{i=0}^{a+2} \cO_{\p2} (2i +b+1)) = 0
		\end{align*}
		because $2(a+2)+b+1 = 2a +b +5 \leq -1$. 
	\end{itemize}
	Therefore $\cE$ is $\mu$-stable when $c=0$ and $m\ge3$.
\end{proof}

The next step is to compute the dimension of $\Ext^1(\cE,\cE)$. First, we need the following technical result.

\begin{lemma}\label{lem:dim_extEE}
	Let $\cE$ be a rank 2 $\mu$-stable locally free sheaf on a nonsingular Fano threefold $X$ with Chern classes $c_1(\cE)$ and $c_2(\cE)$. Then 
	$$ \h1(\cE \otimes \cEv)-\h2(\cE \otimes \cEv) = \frac{1}{2}\big(c_1(\cE)^2 - 4c_2(\cE)\big)K_X-3. $$
\end{lemma}
\begin{proof}
	$\cE$ is $\mu$-stable, thus it is simple by \cite[Corollary 1.2.8]{HL10}, so $\h0(\cE \otimes \cEv)=1$. Serre duality tells us that 
	$\h3(\cE\otimes\cEv)=\h0(\cE\otimes\cEv\otimes K_X)=0$, with the last equality coming from the fact that $-K_X$ is ample, see \cite[Lemma 2.2]{ACG22}.
	
	Next, notice that $c_k(\cE \otimes \cEv)=0$ for $k = 1,3$ and $c_2(\cE \otimes \cEv) = 4c_2(\cE)-c_1(\cE)^2$. Grothendieck--Riemann--Roch theorem then yields
	$$ \chi(\cE \otimes \cEv) = 1-\h1(\cE \otimes \cEv) + \h2(\cE \otimes \cEv) = \frac{(-K_X)(c_1^2 - 4c_2)}{2} +4 $$
	and the claim follows.
\end{proof}

\begin{proposition} \label{prop:dim.mod}
	Let $\cE$ be one of the $h$-instanton bundles on $(X_c,h)$ constructed in Proposition \ref{thm:existence_of_instanton_by_Serre}. Then $\dim \Ext^1(\cE,\cE) = 10\alpha+4c-54$ and $\Ext^2 (\cE, \cE) = 0$.
\end{proposition}
\begin{proof}
	Using $c_1(\cE)=2\xi+3f$ and $c_2(\cE)=\alpha\xi f + 2f^2$, the right hand side of the equality in Lemma \ref{lem:dim_extEE} gives $-(10\alpha+4c-54)$. Therefore, it is enough to show that $\h2(\cE\otimes\cEv)=0$.
	
	To do that, start by considering the cohomology of the sequence \eqref{ses:constructed_bundle_by_Serre} after tensoring by $\cEv\simeq \cE(-2\xi-3f)$; we get the inequality
	\begin{equation}\label{ineq:existence_theorem_3}
		\h2 (\cE \otimes \cEv ) \leq \h2(\cE (-2\xi -f)) + \h2 (\cE\otimes \cI_{Z}(-2f)).
	\end{equation}
	We already showed that $\h2(\cE(-2\xi-f))=0$ in the proof of Lemma \ref{lem:bounds_2ndChern}. Regarding the second summand on of inequality \eqref{ineq:existence_theorem_3}, tensor sequence \eqref{ses:ideal_sheaf} with $\cE (-2f)$ to obtain the following sequence in cohomology
	\begin{equation}\label{ineq:existence_theorem_4}
		\h2 (\cE \otimes \cI_{Z}(-2f)) \leq \h2(\cE (-2f)) + \h1(\cE(-2f)\otimes\cO_Z).
	\end{equation}
	First, notice that $ \h2 (\cE (-2f)) \leq \h2(\cO_{X_c}) + \h2(\cI_{Z \vert X_c}(2\xi -f))$ thanks to sequence in display \eqref{ses:constructed_bundle_by_Serre}. We have $\h2(\cO_{X_c})=0$ and $\h2(\cI_{Z}(2\xi-f)) \leq \h2 (\cO_{X_c} (2\xi -f)) + \h1 (\cO_Z(2\xi-f))$ thanks to sequence \eqref{ses:ideal_sheaf}.
	
	Then $\h2(\cE (-2f))\leq \h2(\cO_{X_c} (2\xi -f)) + \h1 (\cO_Z(2\xi-f))$. We have $ \h2 (\cO_{X_c}(2\xi -f)) = \h2(\p2, S^2 F_c(-1))$ by Lemma \ref{lem:cohomology_of_line_bundles} and it vanishes for all $c$. Also, $\h1(\cO_Z(2\xi-f))=0$ because $Z$ is a disjoint union of $(\alpha-4)$ disjoint curves of class $\xi f$ and $\operatorname{deg}(\xi f\cdot (2\xi -f) ) = 3$. So, $\h2 (\cE (-2f)) = 0$.
	
	Finally, let us consider $\h1(\cE(-2f) \otimes\cO_Z)$. Notice that $\cE(-2f) = \cF$ and the construction of $\cF$ gives the isomorphism $\cF \otimes \cO_Z \simeq \cN_{Z\vert X_c}$. Since $Z$ is a disjoint union of curves $M_i$ of class $\xi f$ and $\cN_{M_i\vert X_c}\simeq\cO_{M_i}(1) \oplus\cO_{M_i}(2)$ as observed right below the exact sequence in display \eqref{ses: resolution of elliptic}, we conclude that 
	$$ \h1 (\cE(-2f) \otimes \cO_Z) = \bigoplus_{i=1}^m\h1(\cO_{M_i}(1)\oplus\cO_{M_i}(2)) = 0. $$ Therefore, $\h2 (\cE \otimes \cI_{Z \vert X_c}(-2f)) =0$ by the inequality in display \eqref{ineq:existence_theorem_4}, and $\h2 (\cE \otimes \cEv ) = 0$ by the inequality in display \eqref{ineq:existence_theorem_3}.
\end{proof}

\begin{remark}
	Comparing the dimension of the family of instanton bundles we just constructed, given by the formula in display \eqref{eq:dim.family}, with the dimension of the irreducible component of the moduli space $\mathcal{M}(2,2\xi+3f,\alpha\xi f+2f^2)$ containing these instanton bundles, obtained in Proposition \ref{prop:dim.mod}, we obtain
	$$ 10\alpha + 4c - 54 - \big(6\alpha-25\big) = 4\alpha + 4c -27; $$
	Note that this is positive when $c\ge1$ and $\alpha\ge6$ and when $c=0$ and $\alpha\ge9$ (which are the conditions under which the dimension formula in display \eqref{eq:dim.family} is valid).
	
	Therefore, we conclude that the family of instanton bundles here constructed is never an open subset of $\mathcal{M}(2,2\xi+3f,\alpha\xi f+2f^2)$.
\end{remark}


\section{Deformation of rank 2 $h$-instanton sheaves on $X_c$} \label{sec:deformation}

In this section, we present the construction of rank 2, orientable, $\mu$-stable instanton bundles on $(X_c,h)$ with $c_2 = a \xi f + bf^2$ for $b>2$, completing the proof of Main Theorem \ref{mthm}.  We will use the elementary transformation of instanton bundles discussed in Proposition \ref{thm:existence_of_instanton_by_Serre} and apply induction on $\beta$. Let us start with a couple of propositions and a remark.

Let $L$ denote a curve in the class of $f^2$ in $A^2(X_c)$. Then $L$ is the pull-back of a point via the map $\pi:X_c\to\mathbb{P}^2$, which is a rational curve. In fact, since $f^2(\xi + f) = 1$, $L$ is a line on $X_c$. If we denote by $\Lambda_L$ the Hilbert scheme of curves in $X_c$ whose class is $f^2$, then $\Lambda_L \simeq \p2$ and distinct elements of $\Lambda_L$ do not intersect each other. Moreover, the structure sheaf $\cO_L$ has the following resolution:
\begin{equation}\label{ses: resolution of line}
	0 \rightarrow \cO_{X_c} (-2f) \rightarrow \cO_{X_c}(-f) \oplus \cO_{X_c}(-f) \rightarrow \cO_{X_c} \rightarrow \cO_L \rightarrow 0.
\end{equation}
In particular, we have $\mathcal{N}_{L|X} \simeq \cO_L^{\oplus2}$.

\begin{proposition}\label{prop: splitting type of serre constructed}
	Let $\cE$ be a bundle constructed in Theorem \ref{thm:existence_of_instanton_by_Serre}. Then the restriction of $\cE$ to a generic line $L \in |f^2|$ splits as $\cO_{L}(1)^{\oplus2} $.
\end{proposition}

\begin{proof}
	We know that $\cE$ fits into the short exact sequence in display \eqref{ses:constructed_bundle_by_Serre}. If we restrict $\cE$ to a generic line $L$ that can be chosen as not intersecting $Z$, then we will have 
	\begin{equation}
		0 \longrightarrow \cO_{L} \longrightarrow \cE \otimes \cO_{L} \longrightarrow \cO_{L}(2) \longrightarrow 0.
	\end{equation}
	We know that $\cE|_L$ must split, so let $\cE \otimes \cO_{L}= \cO_{L}(k) \oplus \cO_{L}(2-k)$ for some $k \le 2$. If $k<0$, the injective map $\cO_{L} \rightarrow \cE \otimes \cO_{L}$ must go into the 2nd summand, so its cokernel would be $ \cO_{L}(k) \oplus T$ where $T$ is a torsion sheaf, rather than $ \cO_{L}(2)$. If $k = 0 \text{ or } 2$, we obtain the trivial extension. So the only possibility left is $k=1$.
\end{proof}

The next step is introducing a family of non-locally free instanton sheaves which is the key to the induction step.

\begin{proposition}\label{prop: kernel instanton sheaf}
	Let $\cE$ be a rank 2, orientable, $\mu$-stable instanton bundle on $(X_c,h)$ with $c_2 = a \xi f + bf^2$ and $\mathrm{Ext}^2(\cE , \cE) = 0$, and let $L\subset X_c$ be a curve of class $|f^2|$ such that $\cE \otimes \cO_{L} \simeq \cO_L(1)^{\oplus 2}$. Then the kernel sheaf $\cE_\varphi$ of the general map  $\varphi: \cE \rightarrow \cO_L(1)$ is a rank 2, orientable, $\mu$-stable instanton sheaf with $c_2 = a \xi f + (b+1)f^2$ satisfying $\cE_\varphi \otimes \cO_{L^{'}} \simeq \cO_{L^{'}}(1)^{\oplus 2}$ for a generic line $L^{'} \in |f^2|$ and $\Ext^2 (\cE_\varphi , \cE_\varphi) = 0$, $\dim \Ext^1 (\cE_\varphi , \cE_\varphi) = \dim \Ext^1 (\cE, \cE) +4.$
\end{proposition}
\begin{proof}
	We have the exact sequence
	\begin{equation}\label{ses:surjection to L}
		0\longrightarrow\cE_\varphi\longrightarrow\cE\longrightarrow \cO_L (1) \longrightarrow0.
	\end{equation}
	Since $ \cO_L (1)$ is a rank 0 instanton sheaf (thanks to Corollary \ref{cor: O_L(1) is rank 0 instanton})), $\cE_\varphi$ is a rank 2 instanton sheaf by Lemma \ref{lem:kernel is instanton sheaf}. Also, by Lemma \ref{lem:kernel is instanton sheaf}, $c_1 (\cE_\varphi)= 4h + K_{X_c}$ and $c_2 (\cE_\varphi) = a \xi f + (b+1)f^2$ since $c_2(\cO_{L} (1)) = f^2$.
	
	We notice that any sheaf $\mu$-destabilizing $\cE_\varphi$ must also $\mu$-destabilize $\cE$. Since we assume that $\cE$ is $\mu$-stable, it follows that $\cE_\varphi$ is $\mu$-stable as well.
	
	Clearly, a general element $L^{'}$ of $\Lambda_L$ does not intersect $L$. Thus, the restriction of sequence \eqref{ses:surjection to L} to each general $L^{'}$ in $\Lambda_L$ remains exact, hence $\cE_\varphi\otimes\cO_{L^{'}} \cong \cE \otimes \cO_{L^{'}} \cong \cO_{L^{'}} (1) ^{\oplus2}$. 
	
	Next, we compute $\dim\Ext^i_{X_c}\big(\cE_\varphi,\cE_\varphi\big)$, for $i\ge1$. By applying the functor $\Hom_{X_c}\big(\cE,-\big)$ to sequence in display \eqref{ses:surjection to L}, we obtain $\Ext^2_{X_c}\big(\cE,\cE_\varphi\big)= \Ext^2_{X_c}\big(\cE,\cE \big) = 0$, because $\Ext^i_{X_c}\big(\cE,\cO_L (1) \big) \cong H^i\big(L,\cO_L^{\oplus2} \big)$, $i\ge1$. Thus, by applying $\Hom_{X_c}\big(-,\cE_\varphi\big)$ to sequence \eqref{ses:surjection to L} we obtain 
	$$
	\Ext^2_{X_c}\big(\cE_\varphi,\cE_\varphi\big) \subseteq \Ext^3_{X_c}\big(\cO_L (1),\cE_\varphi\big).
	$$
	
	Serre duality implies $\Ext^3_{X_c}\big(\cO_L (1) , \cE_\varphi \big) \cong \Hom_{X_c}\big(\cE_\varphi,\cO_L(-1) \big)^{\vee}$. Since $L$ is a locally complete intersection subscheme in $X_c$, we have the isomorphism $\sExt^i (\cO_{L}, \cO_L) \simeq \sExt^{i-1} (\cI_{L}, \cO_{L}) = \bigwedge^i \cN_{L\vert X_c}$. From $\cN_{L\vert X_c} \cong \cO_{\p1}^{\oplus2}$ we deduce  $\sExt^1_{X_c}\big(\cO_L,\cO_L(-1) \big) \cong \cO_L(-1)^{\oplus2}$. By the previous observations, applying $\sHom_{X_c}\big(-,\cO_L(-1)\big)$ to sequence \eqref{ses:surjection to L} we obtain the exact sequence
	\begin{align*}
		0 \longrightarrow \cO_L(-2) \longrightarrow \cO_L(-2)^{\oplus2} \longrightarrow \sHom_{X_c}\big(\cE_\varphi, \cO_L(-1)\big) \longrightarrow \cO_L(-2)^{\oplus2} \longrightarrow0,
	\end{align*}
	because $\sExt^1_{X_c}\big(\cE,\cO_L(-1)\big)=0$. It follows that 
	$$
	\Hom_{X_c}\big(\cE_\varphi,\cO_L(-1)\big)=H^0\big(X_c,\sHom_{X_c}\big(\cE_\varphi,\cO_L(-1)\big)\big)=0,
	$$
	hence $\Ext^2_{X_c}\big(\cE_\varphi,\cE_\varphi\big)=0$. The vanishing $\Ext^3_{X_c}\big(\cE_\varphi,\cE_\varphi\big)=0$ follows from \cite[Lemma 2.2]{ACG22} since $\cE_\varphi$ is $\mu$--stable. Also,  $\Ext^0_{X}\big(\cE_\varphi,\cE_\varphi\big)=1$ because $\cE_\varphi$ is simple ($\mu$--stable sheaves are simple). It follows that
	$$
	\dim\Ext^1_{X_c}\big(\cE_\varphi,\cE_\varphi\big)=1-\chi(\cE_\varphi,\cE_\varphi).
	$$
	
	By applying $\Hom_{X_c}(\cE_\varphi,-\big)$, $\Hom_{X_c}(-,\cE\big)$ and $\Hom_{X_c}(-,\cO_L\big)$ to sequence \eqref{ses:surjection to L} we deduce that
	$$
	\chi(\cE_\varphi,\cE_\varphi)=\chi(\cE,\cE)-\chi(\cE,\cO_L)-\chi(\cO_L,\cE)+\chi(\cO_L,\cO_L).
	$$
	We have $\chi(\cE,\cO_L) = \dim \Ext^0_{X_c} (\cE, \cO_{L}) = 0$ and $\chi(\cO_L,\cE)=- \chi(\cE, \cO_L(-2)) = \dim \Ext^1_{X_c} (\cE, \cO_{L}(-2)) = 4$. Finally, we deduce from the cohomology of sequence \eqref{ses: resolution of line} that $\chi(\cO_L,\cO_L)=0$.
	
	Therefore
	$$
	\dim\Ext^1_{X_c}\big(\cE_\varphi,\cE_\varphi\big) = \dim \Ext_{X_c}^1 (\cE, \cE) +4,
	$$
	as desired.
\end{proof}

\begin{remark}\label{remark: Unique}
	Let $\cE_\varphi$ be the sheaf defined by sequence \eqref{ses:surjection to L}. If $\Gamma\subseteq X_c$ is any subscheme of pure dimension $1$, then
	$$\Ext^1_{X_c}\big(\cO_\Gamma,\cE_\varphi\big)=
	\left\lbrace\begin{array}{ll} 
		0\quad&\text{if $L\not\subseteq \Gamma$,}\\
		1\quad&\text{if $L\subseteq \Gamma$.}
	\end{array}\right. $$
	
	Indeed, by applying $\Hom_{X_c}\big(\cO_\Gamma,-\big)$ to sequence \eqref{ses:surjection to L}, we obtain 
	$$ \Ext^1_{X_c} \big(\cO_\Gamma,\cE_\varphi\big)\cong \Hom_{X_c} \big(\cO_\Gamma,\cO_L\big)\cong H^0\big(X_c, \sHom_{X_c}\big(\cO_\Gamma,\cO_L\big)\big). $$
	We conclude by noticing that
	$$ \sHom_{X_c}\big(\cO_\Gamma,\cO_L\big)=
	\left\lbrace\begin{array}{ll} 
		0\quad&\text{if $L\not\subseteq \Gamma$,}\\
		\cO_L\quad&\text{if $L\subseteq \Gamma$.}
	\end{array}\right. $$
\end{remark}

We are finally ready to state one of the main result of this section, which completes the proof of Main Theorem \ref{mthm}.

\begin{theorem}
	For each $\alpha \ge 6$ when $c\ge1$, $\alpha\ge7$ when $c=0$ and $\beta \geq 2$, there exists rank 2, orientable, $\mu$-stable instanton bundle $\cE$ on $(X_c,h)$ with $c_2 = \alpha \xi f + \beta f^2$. Moreover, $\dim \Ext_{X_c}^1 (\cE, \cE) = 10 \alpha  + 4\beta +4c - 62$ and $\Ext_{X_c}^2 (\cE, \cE) = 0$.
\end{theorem}

\begin{proof}
	We will prove it by induction on $\beta$. The base case $\beta = 2$ has been already proved in Theorem \ref{thm:existence_of_instanton_by_Serre}. 
	
	Now, assume that there exist a $\mu$-stable, orientable instanton bundle $\cE$ of rank 2 with $c_2 = \alpha \xi f + \beta f^2$ such that 
	$$ \dim \Ext^1_{X_c}\big(\cE,\cE\big)=10\alpha+4\beta + 4c -62,\qquad \Ext^2_{X_c}\big(\cE,\cE\big)=\Ext^3_{X_c}\big(\cE,\cE\big)=0. $$
	Thus $\cE$ represents a point in a component $\cM$ of the moduli space of $\mu$-stable sheaves of rank $2$ with Chern classes $c_i(\cE)$, which is smooth at $\cE$ of dimension
	$$
	\dim\Ext^1_{X_c}\big(\cE,\cE\big)=10\alpha+4\beta +4c-62.
	$$
	
	We constructed a sheaf $\cE_\varphi$ in Proposition \ref{prop: kernel instanton sheaf} satisfying all the properties of an instanton bundle, but the one of being locally free. Such a sheaf corresponds to a point in a component $\cM_1$ of the moduli space of $\mu$--stable sheaves of rank $2$ with Chern classes $c_i(\cE_\varphi)$. Moreover, $\cM_1$ is smooth at $\cE_\varphi$ of dimension
	$$
	\dim\Ext^1_{X_c}\big(\cE,\cE\big)=10\alpha+4\beta +4c-58.
	$$
	
	The locus points $\cM_{bad}\subseteq\cM_1$ corresponding to such sheaves are parameterized by the points of $\cM$ (the choice of $\cE$), the points in $\Lambda_L$ (the lines $L$) and elements inside $\bP\big(\Hom_{\p1}\big(\cO_{\p1}(1)^{\oplus2},\cO_{\p1}(1) \big)\big)$ (the morphism $\Hom_{X_c}\big(\cE\otimes\cO_L,\cO_L\big)$ up to scalars). It follows that
	$$
	\dim(\cM_{bad})\le 10\alpha+4\beta +4c-59 < 10 \alpha+4\beta +4c-58 = \dim(\cM_1).
	$$
	
	It follows the existence of a flat family over an integral base $\mathfrak E\to S$ with $\mathfrak E_{s_0}\cong\cE_\varphi$ and $\mathfrak E_s\not\in\cM_{bad}$ for $s\ne s_0$. Thanks to \cite[Satz 3]{BPS80} we can assume $\Ext^2_{X_c}\big(\mathfrak E_s,\mathfrak E_s\big)=0$ for $s\in S$.
	
	For all $s\in S$, we have a natural exact sequence
	\begin{equation}
		\label{seq: bi dual}
		0\longrightarrow\mathfrak E_s\longrightarrow(\mathfrak E_s)^{\vee\vee}\longrightarrow\mathfrak T_s\longrightarrow0.
	\end{equation}
	Then, by Theorem \ref{thm:double_dual_bundle}, $\mathfrak T_s$ is a rank $0$ instanton sheaf and $(\mathfrak E_s)^{\vee\vee}$ is a rank 2 instanton bundle on $X_c$. By semicontinuity we can certainly assume that $\mathfrak E_s$ is a $\mu$--stable, $\h0 \big(\mathfrak E_s (-h)\big)=\h1 \big(\mathfrak E_s(-2h)\big)=0$ and  $\mathfrak E_s\otimes\cO_L\cong\cO_{\p1}(1)^{\oplus2}$ for general $L\in\Lambda_L$. 
	
	If there is $s\in S$ such that the support of $\mathfrak T_s$ is empty, then $\mathfrak E_s\cong(\mathfrak E_s)^{\vee\vee}$. We already know that $\Ext^2_{X_c}\big(\mathfrak E_s,\mathfrak E_s\big)=0$: an easy computation returns also the other dimensions, hence the theorem is proved. 
	
	In the other case, the support of $\mathfrak T_s$ is non-empty for each $s\in S$. Let $c_2(\mathfrak T_s)=\eta_s\xi f+\vartheta_s f^2$ and $\cO_{X_c}(D):=\cO_{X_c}(a\xi+b f)$. If $a,b \ll0$, then we know that $\h{i} \big(\cE (D)\big)=0$ for $i = 1,2$ and $\h0 \big(\cO_L (1+ L \cdot D) \big) = 0$. So, by the short exact sequence \eqref{ses:surjection to L}, we have 
	\begin{equation}\label{middle step of main theorem 1}
		\h2 \big( \cE_\varphi(D) \big) = \h1 \big(\cO_L (1+ L \cdot D) \big) = \h1 \big(\p1, \cO_{\p1} (a + 1) \big) = -a-2.
	\end{equation}
	Similarly, if $a,b \ll0$, then we know that $\h{i} \big((\mathfrak E_s)^{\vee\vee}(D)\big)=0$ for $i = 1,2$ and $\h0 (\mathfrak T_s (D)) = 0$. Also, by semicontinuity, $\h2\big(\mathfrak E_s(D)\big)\le h^2\big(\cE_\varphi(D)\big)$. Then, by the short exact sequence \eqref{seq: bi dual} and the equality \ref{middle step of main theorem 1}, we have 
	\begin{equation}\label{middle step of main theorem 2}
		-\chi \big( \mathfrak T_s(D) \big) = \h1\big(\mathfrak T_s(D)\big)=\h2\big(\mathfrak E_s(D)\big) \le -a-2.
	\end{equation}
	
	Now, one can compute Chern classes of $\cO_{L}(D)$ by using the short exact sequence \eqref{ses: resolution of line}. Then, $c_1 \big(I_{L} (D) \big) = D$, $c_2 \big(I_{L} (D) \big) = f^2$ and $c_3 \big(I_{L} (D) \big) = -D\cdot f^2$.\\
	By the short exact sequence \eqref{ses:ideal_sheaf}, we have $c_1 \big(\cO_{L} (D) \big) = 0$, $c_2 \big(\cO_{L} (D) \big) = -f^2$ and $c_3 \big(\cO_{L} (D) \big) = 2D\cdot f^2$.\\
	Similarly, we have $c_1 \big( \cE_\varphi(D) \big) = 2\xi +3f +2D$, $c_2 \big( \cE_\varphi(D) \big) = c_2 (\cE(D)) +f^2$ and $c_1 \big( \cE_\varphi(D) \big) = 0$ by using the short exact sequence \eqref{ses:surjection to L}.\\
	Since $\mathfrak E_s$ is an element of $\cM_1$, $c_i (\mathfrak E_s) = c_i (\cE_\varphi)$. Moreover, $c_1 \big(\mathfrak E_s^{\vee \vee} (D) \big) = 2\xi + 3f +2D$ and $c_3 \big(\mathfrak E_s^{\vee \vee} (D) \big) = 0$ as $\mathfrak E_s^{\vee \vee}$ is a rank 2 instanton bundle.\\
	Lastly, we have $c_1 \big(  \mathfrak T_s(D) \big) = 0$, $c_2 \big(  \mathfrak T_s(D) \big) = c_2 (\mathfrak E_s^{\vee \vee}) - c_2 (\mathfrak E_s)$ and $c_3 \big(  \mathfrak T_s(D) \big) = - c_1 (\mathfrak E_s (D)) \cdot [c_2 (\mathfrak E_s^{\vee \vee}) - c_2 (\mathfrak E_s)]$ by using the short exact sequence \eqref{seq: bi dual}. Notice that $c_2 \big(  \mathfrak T_s(D) \big)$ is independent on $D$.
	
	Then, using the Riemann--Roch Theorem, we have 
	\begin{align*}
		\chi (\mathfrak T_s(D)) &= \frac{c_3 \big(\mathfrak T_s(D) \big) + \omega_{X_c} \cdot c_2 \big(\mathfrak T_s(D) \big)}{2} = \frac{ c_2 \big(\mathfrak T_s \big) \cdot \big( \omega_{X_c} - c_1 \big(\mathfrak T_s(D) \big) \big)}{2}\\
		&= (2\eta_s + \vartheta_s)(-2-a) +  \eta_s (-2-b).
	\end{align*}
	So, by the inequality \eqref{middle step of main theorem 2}, if $a,b \ll0$, then we have
	\begin{equation*}
		(2\eta_s + \vartheta_s  +1) (2+a) +  \eta_s (2+b) \leq 0.
	\end{equation*}
	Therefore, we have both 
	\begin{equation}\label{middle step of main theorem 3}
		(2\eta_s + \vartheta_s  +1) \geq 0 \text{ and } \eta_s \geq 0
	\end{equation}
	when $a,b \ll0$.\\ Moreover, 
	$$\chi \big( \mathfrak T_s(th) \big) = -\h1 ( \mathfrak T_s(th)) \text{ for } t\leq -2.$$
	So, $(-2-t)(3\eta_s + \vartheta_s) \leq 0 $ for $t \leq -2$. Therefore, we have 
	\begin{equation}\label{middle step of main theorem 4}
		(3\eta_s + \vartheta_s) <0.
	\end{equation}
	If we combine inequalities \eqref{middle step of main theorem 3} and \eqref{middle step of main theorem 4}, then we have $\eta_s = 0$ and $\vartheta_s = -1$. So, $c_2(\mathfrak T_s)= -f^2$ and $c_3(\mathfrak T_s) = (-2\xi -3f) \cdot (-f^2) = 2$. Therefore, $\mathfrak T_s \simeq \cO_{L} (1)$.
	
	It follows the existence of a morphism $S\to \Lambda_L$, hence $\mathfrak T\to S$ turns out to be a flat family. 
	
	The natural sequence
	\begin{equation}
		\label{seqDeformFamily}
		0\longrightarrow\mathfrak E\longrightarrow\mathfrak E^{\vee\vee}\longrightarrow\mathfrak T\longrightarrow0
	\end{equation}
	and the flatness of the families $\mathfrak E\to S$ and $\mathfrak T\to S$ yield the flatness of the induced family $\mathfrak E^{\vee\vee}\to S$. In particular $(\mathfrak E^{\vee\vee})_s\cong(\mathfrak E_s)^{\vee\vee}$ for each $s\in S$, hence $(\mathfrak E^{\vee\vee})_s$ is  a vector bundle for each $s\in S$. It follows that the restriction of sequence \eqref{seqDeformFamily} to $X_c \times\{\ s_0\ \}\subseteq X_c \times S$ must be non--split, hence it coincides with sequence \eqref{ses:surjection to L}, thanks to Remark \ref{remark: Unique}. In particular $(\mathfrak E_{s_0})^{\vee\vee}\cong\cE$, hence $(\mathfrak E_{s})^{\vee\vee}\in\cM$ and, consequently, $\mathfrak E_s\in\cM_{bad}$ for each $s\in S$, contradicting our initial choice. We conclude that the case where the support of $\mathfrak T_s$ is non-empty for each $s\in S$ cannot occur. 
\end{proof}


\section{Another series of instanton bundles on $X_c$} \label{sec:ulrich}

Using a completely different technique, we now present a second infinite series of $h$-instanton bundles on $(X_c,h)$, including a family of Ulrich bundles.

The starting point is the following series of rank 2 bundles on $\mathbb{P}^2$, indexed by positive integers $l\ge1$:
\begin{equation} \label{eq:bdl-g}
	0 \lra G \lra V_{2l+3}\otimes\cO_{\p2}(l+2)\stackrel{\mu}{\lra} V_{2l+1}\otimes\cO_{\p2}(l+3) \lra  0,
\end{equation}
where $V_d$ denotes a $d$-dimensional vector space. The existence of surjective morphisms $\mu$ in $\mathbb{M}:=\Hom(V_{2l+3},V_{2l+1})\otimes\H0(\cO_{\p2}(1))$ is guaranteed by \cite[Proposition 2.1]{Fra00}. It is easy to check that $c_1(G)=3$ and $c_2(G)=l^2+2l+3$.

Set $\mathbb{M}_{0}$ to be the open subset of $\mathbb{M}$ consisting of surjective morphisms. The group $\mathbb{G}:=GL(V_{2l+3})\times GL(V_{2l+1})/\mathbb{C}^*\cdot(\mathbf{1}_{V_{2l+3}},\mathbf{1}_{V_{2l+1}})$ acts on $\mathbb{M}_{0}$ in the obvious manner, so that $\ker(\mu)\simeq\ker(\mu')$ if and only if $\mu$ and $\mu'$ are in the same $\mathbb{G}$ orbit. It follows that the family of isomorphism classes of bundles obtained as the kernel of a morphism $\mu\in\mathbb{M}_0$ is given by the quotient $\mathbb{M}_0/\mathbb{G}$, whose dimension is given by  
$$ \dim(\mathbb{M}_0) - \dim(\mathbb{G}) = 4l(l+2). $$

Dualizing the exact sequence in display \eqref{eq:bdl-g} and tensoring the resulting sequence by $G$ and taking cohomology, one checks that $\H2(G\otimes G^\vee)=0$. Therefore, we get that
$$ \h1(G\otimes G^\vee)=4c_2(G)-c_1(G)^2-3=4l(l+2) $$
meaning that each bundle $G$ as in display \eqref{eq:bdl-g} is a generic point in the moduli space $\cM_{\mathbb{P}^2}(2,3,l^2+2l+3)=\cM_{\mathbb{P}^2}(2,-1,l^2+2l+1)$ of $\mu$-stable rank 2 bundles on $\p2$ with Chern classes $c_1=-1$ and $c_2=l^2+2l+1$.

Let us now see how this series of bundles on $\mathbb{P}^2$ generate $h$-instanton bundle on $(X_c,h)$.

\begin{proposition}
	For each rank 2 bundle $G$ on $\cO_{\p2}$ given by the exact sequence in display \eqref{eq:bdl-g}, the rank 2 bundle $\cG:=\pi^*(G)\otimes\cO_{X_c}(\xi)$ is a $\mu$-stable, orientable $h$-instanton bundle on $(X_c,h)$ with $c_2(\cG)=5\xi f+(l^2+2l+3-c)f^2$ and charge $(l-1)(l+3)$. In particular, $\cG$ is Ulrich when $l=1$.
\end{proposition}

\begin{proof}
	We start by checking that $\cG$ is orientable:
	$$ c_1(\cG) = \pi^*c_1(G) + 2\xi = 2\xi + 3f, $$
	as desired. Note that $\cG$ fits into the following exact sequence
	\begin{equation} \label{eq:instanton-g}
		0 \lra \cG \lra V_{2l+3}\otimes\cO_{X_c}((l+2)f+\xi) \lra V_{2l+1}\otimes\cO_{X_c}((l+3)f+\xi) \lra  0.
	\end{equation}
	from which one can compute that $c_2(\cG)=5\xi f+(l^2+2l+3-c)f^2$, and 
	$\chi(\cG(-h))=-(l-1)(l+3)$; the latter computation uses Lemma \ref{lem:cohomology_of_line_bundles}.
	
	Dualizing the sequence \eqref{eq:instanton-g} and using the isomorphism $\cG\simeq\cG^\vee(3f+2\xi)$, we obtain the exact sequence
	\begin{equation} \label{eq:instanton-g-dual}
		0 \lra V_{2l+1}\otimes\cO_{X_c}(-lf+\xi) \lra V_{2l+3}\otimes\cO_{X_c}((1-l)f+\xi)  \lra \cG \lra  0.
	\end{equation}
	
	To see that $\cG$ is an $h$-instanton, note that
	$$ \h0(\cG(-h)) \le (2l+3)\cdot\h0(\cO_{X_c}(-lf)) + (2l+1)\cdot\h1(\cO_{X_c}(-(l+1)f)) \quad {\rm and} $$
	$$ \h1(\cG(-2h)) \le (2l+3)\cdot\h1(\cO_{X_c}(-(l+1)f-\xi) + (2l+1)\cdot\h2(\cO_{X_c}(-(l+2)f-\xi)). $$
	Using Lemma \ref{lem:cohomology_of_line_bundles} one quickly checks that $\h0(\cG(-h))=\h1(\cG(-2h))=0$, thus $\cG$ is an $h$-instanton bundle.
	
	Finally, using that fact that $\pi_*$ is a right adjoint to $\pi^*$, we note that 
	\begin{align*}
		\Hom_{X_c}(\cG,\cO_{X_c}(a\xi+bf)) & \simeq \Hom_{\mathbb{P}^2}(G,\pi_*\cO_{X_c}((a-1)\xi+bf)) \\
		&\simeq \Hom_{\mathbb{P}^2}(G,(S^{a-1}F_c)(b)).
	\end{align*}
	We must argue that this group vanishes when $\deg(\cO_{X_c}(a\xi+bf))\le\deg(\cG)$, or equivalently
	\begin{equation}\label{eq:ineq2}
		a(9-c)+4b \le 15-c
	\end{equation}
	Since $\pi_*\cO_{X_c}((a-1)\xi+bf)=0$ when $a\le0$, we get that  $\Hom_{X_c}(\cG,\cO_{X_c}(a\xi+bf))=0$ when $a\le0$; so assume from now on that $a\ge1$. 
	
	When $c\ge1$, the rank 2 bundle $F_c$ is $\mu$-semistable and therefore $S^pF_c$ is also $\mu$-semistable with slope $p\cdot \mu(F_c)=p$. It follows that $\Hom_{\mathbb{P}^2}(G,(S^{a-1}F_c)(b))=0$ when $\mu(G)>\mu((S^{a-1}F_c)(b))$, equivalently, $\frac{3}{2}>a-1+b$. Note that the inequality in display \eqref{eq:ineq2} implies that 
	$$ 4a+4b \le 15-c + (c-5)a \le 10 ~~ \Longleftrightarrow ~~ a+b-1\le \frac{3}{2} ; $$
	but $a$ and $b$ are integers, so we have $a+b-1 < \frac{3}{2}$, as desired.
	
	When $c=0$, then $S^{a-1}F_c=\bigoplus_{i=0}^{a-1}\cO_{\p2}(2i)$. It follows that 
	$$ \Hom_{\mathbb{P}^2}(G,(S^{a-1}F_c)(b)) = \bigoplus_{i=0}^{a-1}\H0 (G(2i+b-3)); $$
	since $G$ is $\mu$-stable with slope $\frac{3}{2}$, we get that $\H0 (G(2a+b-5))=0$ provided $-(2a+b-5)\le\frac{3}{2}$, equivalently $2a+b\le\frac{7}{2}$. The inequality in display \eqref{eq:ineq2} simplifies to $8a+4b\le15-a\le14$, which is coincides with $2a+b\le\frac{7}{2}$, as desired.
\end{proof}

Finally, we argue that $\cG$ is a smooth point in its moduli space.

\begin{proposition}
	$\h1(\cG\otimes\cG^\vee)=4l(l+2)$ and $\h2(\cG\otimes\cG^\vee)=0$
\end{proposition}
\begin{proof}
	The formula in Lemma \ref{lem:dim_extEE} yields, in the case at hand,
	$$ \h1(\cG\otimes\cG^\vee) - \h2(\cG\otimes\cG^\vee) = 10\cdot5+4(l^2+2l+3-c)+4c-62=4l(l+2). $$
	Therefore, it is enough to show that $\h2(\cG\otimes\cG^\vee)=0$.
	
	Twisting the dual of sequence in display \eqref{eq:instanton-g} by $\cG$ we obtain
	$$ 0 \lra V_{2l+1}\otimes\cG(-(l+3)f-\xi) \lra V_{2l+3}\otimes\cG(-(l+2)f-\xi)  \lra \cG\otimes\cG^\vee \lra  0, $$
	so that
	$$ \h2(\cG\otimes\cG^\vee) \le (2l+3)\cdot\h2(\cG(-(l+2)f-\xi)) + (2l+1)\cdot\h3(\cG(-(l+3)f-\xi)). $$
	Next, twist \eqref{eq:instanton-g} by $\cO_{X_c}(-(l+s)f-\xi)$ with $s=2,3$ to get
	$$ 0 \lra \cG(-(l+s)f-\xi) \lra V_{2l+3}\otimes\cO_{X_c}((2-s)f) \lra V_{2l+1}\otimes\cO_{X_c}((3-s)f) \lra  0. $$
	It follows: 
	$$ \h2(\cG(-(l+2)f-\xi)) \le (2l+1)\cdot\h{1}(\cO_{X_c}(f)) + (2l+3)\cdot\h2(\cO_{X_c}) \quad {\rm and}$$
	$$ \h3(\cG(-(l+3)f-\xi)) \le (2l+1)\cdot\h{2}(\cO_{X_c}) + (2l+3)\cdot\h3(\cO_{X_c}(-f)). $$
	Using Lemma \ref{lem:cohomology_of_line_bundles} one easily checks that
	$$ \h{1}(\cO_{X_c}(f))=\h{2}(\cO_{X_c})=\h3(\cO_{X_c}(-f))=0, $$
	thus $\h2(\cG\otimes\cG^\vee)=0$, as desired.
\end{proof}

We want to close this exposition with some open problems which are worth pondering.

\begin{open problem}
	We show in Main Theorem 1 the existence of $h$-instanton bundles with second Chern class $\alpha \xi f + \beta f^2$ for $\beta \geq 2$. However, we know that $\beta$ can take values less than 2 in principle, thanks to Lemma \ref{lem:bounds_2ndChern}. Therefore the following problem is open: to construct rank 2, orientable, $\mu$-stable $h$-instanton bundles on $(X_c,h)$ with $\beta \leq 1$.
\end{open problem}

\begin{open problem}
	The definition of $h$-instanton bundle is directly dependent on the choice of polarization of the ambient variety. In this work, we only worked with the polarization $h = \xi + f$. Therefore, similar research can be carried out with another choice of polarization for $X_c$. In particular, the existence of rank 2, orientable, $\mu$-stable $h$-instanton bundle on $(X_c,h)$ where $h$ is different from $\xi + f$ is open.
\end{open problem}

\begin{open problem}
	Similarly, the existence of $h$-instanton bundles with defect $\delta=1$ as defined in \cite[Definition 1.3]{AC23}, is also open.
\end{open problem}

\begin{open problem}
	The present article was dedicated to the existence of rank 2 $h$-instanton bundles on $(X_c,h)$. It is easy to check that extensions of $h$-instanton sheaves are also $h$-instanton, so we have also obtained the existence of $\mu$-semistable $h$-instanton bundles of even rank on $X_c$. However, the construction of $\mu$-stable $h$-instanton sheaves of arbitrary rank is still an open challenge.
\end{open problem}

\begin{open problem}
	In this work, we study orientable, rank 2 $h$-instanton bundles, which have first Chern class $4h+K_{X_c}$. However, in general, a rank 2 $h$-instanton bundle $\cE$ has the first Chern class satisfying the equation (see \cite[Theorem 1.6]{AC23})
	\[
	c_1(\cE) h^2 = 4h^3 + K_{X_c}h^2.
	\]
	Therefore, one can study rank 2 $h$-instanton bundles in general, without restricting orientable ones on a variety that has Picard rank bigger than one, like in our case. For example, there is another potentially interesting family of rank 2 locally free sheaves on $X_c$ that can be constructed from the curves $Z$:
	$$ 0 \lra \cO_{X_c} \lra \mathcal{H} \lra \cI_{Z}(\xi+f) \lra 0. $$
	Note that $\mathcal{H}=\cO_X(\xi)\oplus\cO_X(f)$ when $m=1$. However, since $c_1(\mathcal{H})=\xi+f$, there are no $a,b$ such that 
	$$ c_1(\mathcal{H}(a\xi+bf))=(2a+1)\xi+(2b+1)f = 4h+K_X. $$
	So there is no chance of producing orientable instantons for the polarization $h = \xi + f$. Even if we change the polarization, we will have the same problem for rank 2 orientable instantons. So, is it possible to find a different polarization for $X_c$ such that $\mathcal{H}$ is a non-orientable, rank 2 $h$-instanton bundle up to a twist of a line bundle?
\end{open problem}

\bigskip


\bigskip
\bigskip
\noindent
Ozhan Genc,\\
Department of Mathematics and Informatics, Jagiellonian University,\\
ul. prof. Stanisława Łojasiewicza 6, \\
30-348 Krak{\'o}w, Poland\\
e-mail: {\tt ozhangenc@gmail.com}

\bigskip
\noindent
Marcos Jardim, \\
Universidade Estadual de Campinas (UNICAMP) \\
Instituto de Matemática, Estatística e Computação Científica (IMECC) \\ 
Departamento de Matemática \\
Rua Sérgio Buarque de Holanda, 651,\\
13083-859 Campinas-SP, Brazil\\ 
e-mail: {\tt jardim@unicamp.br}

\end{document}